\documentclass[11pt]{amsart}

\usepackage[breaklinks]{hyperref}
\hypersetup{%
  pdftitle={\@title},            % title
  pdfauthor={\@author},          % author
  unicode=true,           % non-Latin characters in Acrobat’s bookmarks
  colorlinks=true,       % false: boxed links; true: colored links
  linkcolor=blue,         % color of internal links
  citecolor=blue,      % color of links to bibliography
  filecolor=green,        % color of file links
  urlcolor=blue           % color of external links
}

\usepackage[utf8]{inputenc}
\usepackage[T1]{fontenc}
\usepackage[english]{babel}
\usepackage[top=1in, bottom=1.25in, left=1.25in, right=1.25in]{geometry}

\usepackage[usenames,dvipsnames]{xcolor}

\usepackage[round]{natbib}

\usepackage{amsmath,amsfonts,amssymb,amsthm,amscd}

\newtheorem{theorem}{Theorem}
\newtheorem{proposition}{Proposition}
\newtheorem{lemma}{Lemma}

\usepackage{nameref}
\usepackage[nameinlink]{cleveref}
\crefname{equation}{equation}{equations}
\crefname{assumption}{assumption}{assumptions}
\creflabelformat{enumi}{(#2#1#3)}

\usepackage{mathtools}

\providecommand\given{}
\newcommand\SetSymbol[1][]{
  \nonscript\,#1:\nonscript\,\mathopen{}\allowbreak}
\DeclarePairedDelimiterX\Set[1]{\lbrace}{\rbrace}%
{ \renewcommand\given{\SetSymbol[]} #1 }

\allowdisplaybreaks[1]

\usepackage{stmaryrd}
\usepackage{mathrsfs}
\usepackage[inline]{enumitem}
\usepackage{bbm}

\usepackage{amsaddr}

\title[BNP estimation for QHT]{Bayesian nonparametric estimation for Quantum Homodyne Tomography}
\author{Zacharie Naulet}
\address{CEREMADE, Université Paris-Dauphine, France}
\email{zacharie.naulet@dauphine.eu}
\author{{\'E}ric Barat}
\address{CEA, LIST, Laboratory of Modeling, Simulation and Systems\\
  F-91191 Gif-sur-Yvette, France}
\date{\today}

\begin{document}

\begin{abstract}
  We estimate the quantum state of a light beam from results of quantum homodyne
tomography noisy measurements performed on identically prepared quantum
systems. We propose two Bayesian nonparametric approaches. The first approach is
based on mixture models and is illustrated through simulation examples. The
second approach is based on random basis expansions. We study the theoretical
performance of the second approach by quantifying the rate of contraction of the
posterior distribution around the true quantum state in the $L^2$ metric.

\end{abstract}

\maketitle

\section{Introduction}
\label{sec:introduction}

\textit{Quantum Homodyne Tomography} (QHT), is a technique for reconstructing
the quantum state of a monochromatic light beam in cavity
\citep{ArtilesGillothers2005}. Unlike classical optics, the predictions of
quantum optics are probabilistic so that we cannot in general infer the result
of a single measurement, but only the distribution of possible outcomes. The
quantum state of a monochromatic light beam in cavity is a positive,
self-adjoint and trace-class operator $\rho$ acting on the Hilbert space
$L^2(\mathbb R)$. We should here distinguish the \textit{pure states} which are
projection operators onto one-dimensional subspaces of $L^2(\mathbb R)$, and
\textit{mixed-states} which are all the other possible states.

Having prepared a quantum system in state $\rho$, the aim of the physicist is to
perform \textit{measurement} of certain \textit{observables}. Mathematically
speaking, an observable $A$ is a self-adjoint operator on $L^2(\mathbb R)$. A
measurement is a mapping which assigns to an observable $A$ and a state $\rho$ a
probability measure $\mu_A$ on $\mathbb R$; this mapping is given by the
so-called \textit{Born-von Neumann formula} \citep{Hall2013}.

Two observables of interest in quantum optics correspond to the measurements of
the \textit{electric field} and the \textit{magnetic field} of a light beam, and
are given respectively by the operator $\mathbf Q$ and $\mathbf P$ with domains
$D(\mathbf Q) := \Set{\psi \in L^2(\mathbb R) \given x\mapsto x\psi(x) \in
  L^2(\mathbb R)}$ and $D(\mathbf P) := \Set{\psi \in L^2(\mathbb R) \given
  x\mapsto \psi'(x) \in L^2(\mathbb R)}$. The operarors $\mathbf Q$ and $\mathbf
P$ act on $D(\mathbf Q)$, respectively $D(\mathbf P)$, as
\begin{equation*}
  \mathbf Q\psi(x) = x\psi(x),\ \mathrm{and}\qquad \mathbf P\psi(x) = -i
  \psi'(x).
\end{equation*}
The derivative in the definitions of $D(\mathbf P)$ and $\mathbf P$ is
understood in the distributional sense.

By virtue of the \textit{Heisenberg uncertainty principle} \citep{Hall2013}, the
observables $\mathbf P$ and $\mathbf Q$ cannot be measured simultaneously; that
is there is no joint probability distribution associated to the simultaneous
measurement of $\mathbf P$ and $\mathbf Q$. Nevertheless, the \textit{Wigner
  density} $W_{\rho} : \mathbb R^2 \rightarrow \mathbb R$, with respect to the
Lebesgue measure on $\mathbb R^2$, as defined below, is the closest object to a
joint probability density function associated to the joint measurement of
$\mathbf P$ and $\mathbf Q$ on a system in state $\rho$. The Wigner distribution
satisfies $\int_{\mathbb R^2}W_{\rho} = 1$, and its marginals on any direction
are \textit{bona-fide} probability density functions. In general, however,
$W_{\rho}$ fails to be a proper joint probability density function, as it can
take negative values, reflecting the non classicality of the quantum state
$\rho$. For a pure state $\rho_{\psi}$, $\psi \in L^2(\mathbb R)$, the Wigner
quasi-probability density of $\rho_{\psi}$ is defined as
\begin{equation}
  \label{eq:11}
  W_{\psi}(x,\omega):= \int_{\mathbb R}
  \psi(x+t/2)\overline{\psi(x-t/2)} e^{-2\pi i\omega t}dt,\qquad (x,\omega) \in
  \mathbb R^2.
\end{equation}
We delay to later the definition of the Wigner distribution for mixed states,
which will follow from the definition for pure states in a relatively
straightforward fashion. Here we take the opportunity to say that whenever we
will be concerned with pure states, we will identify the state $\rho_{\psi}$ to
the function $\psi \in L^2(\mathbb R)$, and talk abusively about the state
$\psi$.

Although we cannot measure simultaneously the observables $\mathbf P$ and
$\mathbf Q$, it is possible to measure the \textit{quadrature observables},
defined as $\mathbf X_{\theta} := \mathbf Q \cos\theta + \mathbf P\sin \theta$
for all $\theta \in [0,\pi]$. We denote by $X_{\theta}^{\rho}$ the random
variable whose distribution is the measurement of $\mathbf X_{\theta}$ on the
quantum system in state $\rho$. Assuming that $\theta$ is drawn uniformly from
$[0,\pi]$, the joint probability density function (with respect to the Lebesgue
measure on $\mathbb R \times [0,\pi]$) for $(X_{\theta}^{\rho},\theta)$ is
given by the \textit{Radon transform} of the Wigner distribution $W_{\rho}$,
that is
\begin{equation}
  \label{eq:18}
  p_{\rho}(x,\theta)
  :=
  \frac{1}{\pi}
  \int_{\mathbb R} W_{\rho}(x \cos \theta-\xi \sin\theta,x\sin\theta + \xi
  \cos\theta)\,d\xi,\quad (x,\theta)\in \mathbb R \times [0,\pi].
\end{equation}
For a pure state $\psi \in L^2(\mathbb R)$, there is a convenient way of
rewriting the previous equation, as stated for example in
\cite[equation~4.14]{MarkusBryanJorge2010},
\begin{equation}
  \label{eq:1}
  p_{\psi}(x,\theta)
  =
  \begin{cases}
    \frac{1}{2\pi |\sin\theta|} \left|
      \int_{\mathbb R}\psi(z)\,\exp\left( \pi i  \frac{\cot\theta}{2} z^{2} -
        \pi i \frac{x z}{\sin\theta} \right)\, dz
    \right|^{2} & \theta \ne 0,\ \theta \ne \pi/2,\\
    | \psi(x) |^2 / \pi & \theta = 0,\\
    | \widehat{\psi}(x) |^2 / \pi & \theta = \pi/2,
  \end{cases}
\end{equation}
where $\widehat{\psi}$ is the Fourier transform of $\psi$ (according to the
convention defined in the next section of the paper). \Cref{eq:1} emphasizes
that for any $(x,\theta)$ we indeed have $p_{\psi}(x,\theta) \geq 0$, a fact
that remains true for mixed states, but which is not so clear from the
definition of \cref{eq:18}.

Quantum homodyne tomography is an experiment that allow for measuring the
quadrature observables $\mathbf X_{\theta}$ for a monochromatic light beam in
cavity in state $\rho$. Here we consider the situation when we perform identical
and independent measurements of $\mathbf X_{\theta}$ on $n$ quantum systems in
the same state $\rho$, with $\theta$ spread uniformly over $[0,\pi]$. Following
\citet{ButuceaGutaArtilesEtAl2007}, it turns out that a good model for a
realistic quantum homodyne tomography must take into account noise on
observations.

In practice, the noise is mostly due to the fact that a number of photons fails
to be detected. The ability of the detector to detect photons is quantified by a
parameter $\eta \in [0,1]$, called the \textit{efficiency} of the detector. When
$\eta = 0$, then the detector fails to detect all photons, whereas $\eta = 1$
corresponds to the ideal case where all the photons are detected. In general, it
is assumed that $\eta$ is known ahead of the measurement process, and $\eta$ is
relatively close to one, according to the physicists. Then, from
\citet[section~2.4]{ButuceaGutaArtilesEtAl2007}, a more realistic model for
quantum homodyne tomography is to consider that we observe the random variables
(given $\theta$)
\begin{equation*}
  Y_{\theta}^{\rho} = X_{\theta}^{\rho} + \sqrt{\frac{1 - \eta}{\eta}}
  X_{\theta}^{\mathrm{vac}},
\end{equation*}
where $X_{\theta} \sim p_{\rho}(\cdot \mid \theta)$, and
$X_{\theta}^{\mathrm{vac}}$ is the random variable whose distribution is the
measurement of $\mathbf X_{\theta}$ on the \textit{vacuum state} and is assumed
independent of $X_{\theta}^{\rho}$. Here we adopt the convention that the vacuum
state is the projection operator onto $x \mapsto 2^{-1/4}\exp(-\pi x^2)$. It
turns out from \cref{eq:11,eq:2} that $X_{\theta}^{\mathrm{vac}}$ has a normal
distribution with mean zero and variance\footnote{Some readers may have noticed
  that the variance here is different that in
  \citet{ButuceaGutaArtilesEtAl2007}. This comes from a different convention for
  defining the vacuum state.} $1/(4\pi)$. This leads to the following
\textit{efficiency corrected} probability density function of observations,
\begin{equation}
  \label{eq:2}
  p_{\psi}^{\eta}(y,\theta)
  := \sqrt{\frac{2}{1-\eta}} \int_{\mathbb R}
  p_{\psi}(x,\theta)
  \exp\left[-\frac{2\pi\eta}{1-\eta}\left( x - y \right)^2 \right]\,dx.
\end{equation}
To shorten notations, we define
\begin{equation}
  \label{eq:12}
  \gamma := \frac{\pi (1-\eta)}{2\eta},\ \mathrm{and}\quad
  G_{\gamma}(x) := \sqrt{\pi/\gamma} \exp\left[- \pi^2 x^2 / \gamma\right],
\end{equation}
so that we have
$p_{\psi}^{\eta}(y,\theta) = [p_{\psi}(\cdot,\theta) * G_{\gamma}](y)$, where
$*$ denote the convolution product.

To summarize the statistical model we are considering in this paper, we aim at
estimating the Wigner density function $W_{\rho}$, or better directly the state
$\rho$, from $n$ independent and indentically distributed noisy observations
$(Y_1,\theta_1),\dots,(Y_n,\theta_n)$ distributed according to the distribution
that has the density function of \cref{eq:2} with respect to the Lebesgue
measure on $\mathbb R \times [0,\pi]$.

The problem of QHT is a statistical nonparametric \textit{ill-posed} inverse
problem that has been relatively well studied from a frequentist point of view
in the last few years, and now quite well understood. We mention here only
papers with theoretical analysis of the performance of their estimation
procedure. We should classify frequentist methods in two categories, depending
on whether they are based on estimating the state $\rho$, or estimating
$W_{\rho}$ (although $\rho \mapsto W_{\rho}$ is one-to-one, methods based on
estimating $W_{\rho}$ don't permit to do the reverse path from
$W_{\rho} \mapsto \rho$).

The estimation of the state $\rho$ from QHT measurements has been considered in
the ideal situation ($\eta = 1$, no noise) by \citet{ArtilesGillothers2005},
while the noisy setting is investigated in \citet{AubryButuceaMeziani2008} under
Frobenius-norm risk. For smoothness class of realistic states
$\mathcal{R}(C,B,r)$, an adaptive estimation procedure has been proposed by
\citet{AlquierMezianiPeyre2013} and an upper bound for the Frobenius-norm risk
is given. Goodness-of-fit testing is investigated in \citet{Meziani2008}.

Regarding frequentist methods for estimating $W_{\rho}$, the first result goes
back to \citet{GutaArtiles2007}, where sharp minimax results are given over a
class of smooth Wigner functions $\mathcal{A}(\beta,r = 1,L)$, under the
pointwise risk. The noisy framework has been considered in
\citet{ButuceaGutaArtilesEtAl2007}; authors obtain the minimax rates of
convergence under the pointwise risk and propose an adaptive estimator over the
set of parameters $\beta > 0$, $r \in (0,1)$ that achieve nearly minimax
rates. In the same time \citet{Meziani2007} explored the estimation of a
quadratic functional of the Wigner function, as an estimator of the purity of
the state. In, \citet{AubryButuceaMeziani2008} an upper bound for the $L^2$-norm
risk over the class $\mathcal{R}(C,B,r)$ is given. More recently,
\citet{LouniciMezianiPeyre2015} established the first sup-norm risk upper bound
over $\mathcal{A}(\beta,r,L)$, as well as the first minimax lower bounds for
both sup-norm and $L^2$-norm risk; they also provide an adaptive estimator that
achieve nearly minimax rates for both sup-norm and $L^2$-norm risk over
$\mathcal{A}(\beta,r,L)$ for all $\beta > 0$ and $r \in (0,2)$.

To our knowledge, no Bayesian nonparametric method has been proposed to address
the problem of QHT with noisy data, a gap that we try to fill with this
paper. In particular, after having introduced preliminary notions in the next
section, we propose two families of prior distributions over pure states that
can be useful in practice, namely \textit{mixtures of coherent-states} and
\textit{random Wilson series}. Regarding mixed-states, we will discuss how we
can straightforwardly extend the prior distributions over pure states onto prior
distributions over mixed states. After presenting simulation results, we will
investigate posterior rates of contraction for random Wilson series in the main
section of the paper. Rates of contraction, or even consistency, is still
challenging for coherent states mixtures, a fact that will be discussed more
thoroughly in \cref{sec:assumptions}.

\section{Preliminaries}
\label{sec:preliminaries}

\subsection{Notations}
\label{sec:notations}

For $x,y \in \mathbb R^d$, $x y$ denote the euclidean inner product of $x$ and
$y$, and $\|x\|$ is the euclidean norm of a vector $x \in \mathbb R^d$. For any
function $f$, we denote by $\breve{f}$ the involution $\breve{f}(x) = f(-x)$. We
use the notation $\|\cdot\|_p$ for the norm of the spaces $L^p(\mathbb R^d)$.

We use the following convention for the Fourier transform of a function
$f \in L^1(\mathbb R^d)$.
\begin{equation*}
  \mathscr{F}f (\omega) := \widehat{f}(\omega) := \int_{\mathbb R^d}f(x)e^{-2\pi
    i x\omega}\,dx,\qquad \forall \omega \in \mathbb R^d.
\end{equation*}
Then, whenever $f \in L^1(\mathbb R^d)$ and $\mathscr{F}f \in L^1(\mathbb R^d)$,
the inverse Fourier transform $\mathscr{F}^{-1}\mathscr{F}f = f$ is well defined
and given by
\begin{equation*}
  f(x) =
  \int_{\mathbb R^d} \widehat{f}(\omega) e^{2\pi i \omega x}\,d\omega,\qquad
  \forall x\in \mathbb R^d.
\end{equation*}

Regarding the space $L^2(\mathbb R^d)$, we use the convention that the inner
product
$\langle \cdot,\, \cdot \rangle : L^2(\mathbb R^d)\times L^2(\mathbb R^d)
\rightarrow \mathbb C$ is linear in the first argument and antilinear in the
second argument, that is for two functions $f,g \in L^2(\mathbb R^d)$ we define
$ \langle f,\, g\rangle := \int_{\mathbb R^d}f(x) \overline{g(x)}\, dx$, where
$\overline{z}$ is the complex conjugate of $z \in \mathbb C$. The \textit{unit
  circle} of $L^2(\mathbb R^d)$ will be denoted by $\mathbb S^2(\mathbb R^d)$;
that is
$\mathbb S^2(\mathbb R^d) := \Set{f \in L^2(\mathbb R^d) \given \|f\|_2 = 1}$.

We shall sometimes encounter the \textit{Schwartz} space
$\mathcal{S}(\mathbb R^d)$; that is the space of all infinitely differentiable
functions $f : \mathbb R^d \rightarrow \mathbb R$ for which
$|x^{\alpha}D^{\beta}f(x)| < +\infty$ for all $\alpha,\beta \in \mathbb N^d$,
with the convention $x^{\alpha} = x_1^{\alpha_1}\dots x_d^{\alpha_d}$ and
$D^{\beta}f = \partial^{\beta_1+\dots+\beta_d}f
/(\partial x_1^{\beta_1}\dots \partial x_d^{\beta_d})$.

Dealing with probability distributions, we consider the \textit{Hellinger}
distance
$H^2(P,Q) := \frac 12 \int ( \sqrt{dP/ d\lambda} - \sqrt{dQ/d\lambda} )^2\,
d\lambda$, for any probability measures $P,Q$ absolutely continuous with respect
to a common measure $\lambda$.

We denote by $P_{\rho}$, respectively $P_{\rho}^{\eta}$, the distributions that
admit \cref{eq:18}, respectively \cref{eq:2}, as density with respect to the
Lebesgue measure on $\mathbb R \times [0,\pi]$. When $\rho \equiv \rho_{\psi}$
denote a pure state, we denote the previous distribution by $P_{\psi}$ and
$P_{\psi}^{\eta}$, respectively.

Finally, inequalities up to a generic constant are denoted by the symbols
$\lesssim$ and $\gtrsim$, where $a \lesssim b$ means $a \leq C b$ for a constant
$C > 0$ with no consequence on the result of the proof.

\subsection{Coherent states}
\label{sec:coherent-states}

In quantum optics, a coherent state refers to a state of the quantized
electromagnetic field that describes a classical kind of behavior.

Let $T_xf(y) := f(y - x)$, $M_{\omega}f(y) = e^{2\pi i \omega y}f(y)$, denote
the translation and modulation operators, respectively, and $g$ a window
function with $\|g\|_2 = 1$; most of time $g$ is chosen as
$g(x) = 2^{-1/4}\exp(-\pi x^2)$. Mathematically speaking, coherent states are
pure states $\rho_{\psi}$, that is projection operators, described by a
\textit{wave-function} $\psi$ belonging to
\begin{equation*}
  \Set*{ \psi \in L^2(\mathbb R) \given \psi = T_xM_{\omega}g\quad (x,\omega)
    \in \mathbb R^2}.
\end{equation*}

Note that the operators $T_x$ and $M_{\omega}$ are isometric on
$L^p(\mathbb R^d)$ and $\|f\|_p = \|T_xM_{\omega}f\|_p$ for any
$1 \leq p \leq \infty$, all $f\in L^p(\mathbb R^d)$ and all
$x,\omega \in \mathbb R$.

\subsection{Wilson bases}
\label{sec:wilson-bases}

\citet{DaubechiesJaffardJourne1991} proposed simple \textit{Wilson bases} of
exponential decay. They constructed a real-valued function $\varphi$ such that
for some $a,b > 0$,
\begin{equation*}
  |\varphi(x)| \lesssim e^{-a|x|},\qquad |\widehat{\varphi}(\omega)| \lesssim
  e^{-b|\omega|},
\end{equation*}
and such that the $\varphi_{lm}$, $l \in \mathbb N$, $m\in \frac 12 \mathbb Z$
defined by
\begin{equation*}
  \varphi_{lm}(x) :=
  \begin{cases}
    \varphi(x - 2m) & \mathrm{if}\ l=0,\\
    \sqrt{2}\varphi(x - m)\cos(2\pi l x) & \mathrm{if}\ l\ne 0\ \mathrm{and}\ 2m
    + l\ \mathrm{is\ even},\\
    \sqrt{2}\varphi(x - m)\sin(2\pi l x) & \mathrm{if}\ l\ne 0\ \mathrm{and}\ 2m
    + l\ \mathrm{is\ odd},
  \end{cases}
\end{equation*}
constitute an orthonormal base for $L^2(\mathbb R)$. Following
\citet[section~8.5]{Groechenig}, we may rewrite $\varphi_{lm}$ in a convenient
form for the sequel, emphasizing the relationship with coherent states,
\begin{equation}
  \label{eq:32}
  \varphi_{lm} = c_lT_m(M_l + (-1)^{2m+l}M_{-l})\varphi,\qquad (l,m)\in \mathbb
  N \times \textstyle\frac{1}{2}\mathbb Z,
\end{equation}
where $c_0 := 1/2$ and $c_l := 1/\sqrt{2}$ for $l\geq 1$.

\section{Prior distributions}
\label{sec:prior-distribution}

We recall that a pure state $\rho_{\psi}$ is a projection operator onto a
one-dimensional subspace of $L^2(\mathbb R)$. Before giving the methodology for
estimating general states, we introduce two types of prior distribution over
pure-states. More precisely, we first define two probability distributions over
$\mathbb S^2(\mathbb R)$, that can be trivially identified with the set of
pure-state through the mapping
$\mathbb S^2(\mathbb R) \ni \psi \mapsto \rho_{\psi}$; then we will show how to
enlarge these prior distributions to handle mixed states.

The first prior model is based on Gamma mixtures, whereas the second is based on
the Wilson base of exponential decay.

\subsection{Gamma Process mixtures of coherent states}
\label{sec:gamma-proc-mixt}

For any finite positive measure $\alpha$ on the measurable space
$(X,\mathcal{X})$, let $\Pi_{\alpha}$ denote the Gamma process distribution with
parameter $\alpha$; that is, a $Q \sim \Pi_{\alpha}$ is a measure on
$(X,\mathcal{X})$ such that for any disjoints $B_1,\dots,B_k \in \mathcal{X}$
the random variables $Q(B_1),\dots,Q(B_k)$ are independent random variables with
distributions $\mathrm{Ga}(\alpha(B_i),1)$, $i=1,\dots,k$.

We suggest a mixture of coherent states as prior distribution on the wave
function $\psi$. For a Gamma random measure $Q$ on
$\mathbb R^2 \times [0,2\pi]$, our model may be summarized by the following
hierarchical representation.  Recall that $P_{\psi}^{\eta}$ denote the
probability distribution having the density of \cref{eq:2}, with
$\rho = \rho_{\psi}$ the projection operator onto $\psi$.
\begin{gather*}
  (Y_1,\theta_1),\dots,(Y_n,\theta_n) \overset{\mathrm{i.i.d}}{\sim}
  P_{\psi}^{\eta},\quad \mathrm{with}\ \psi = \widetilde{\psi} /
  \|\widetilde{\psi}\|_2\\
  \widetilde{\psi}(z) = \int_{\mathbb R^2\times [0,2\pi]}e^{i
    \phi}T_xM_{\omega}g(z)\,Q(dxd\omega d\phi)\\
  Q \sim \Pi_{\alpha}.
\end{gather*}

\subsection{Random Wilson series}
\label{sec:random-gabor-frames}

Let $(\varphi_{lm})$ be the orthonormal Wilson base with exponential decay of
\cref{sec:wilson-bases}. For any positive number $Z$, let $\Lambda_Z$ be the
spherical array
\begin{equation*}
  \Lambda_Z := \Set*{(l,m) \in \mathbb N \times \textstyle \frac 12 \mathbb Z
    \given l^2 +  m^2 < Z^2}.
\end{equation*}
Also define the simplex $\Delta_Z$ in the $\ell_2$ metric as
\begin{equation*}
  \Delta_Z := \Set*{ \mathbf p = (p_{lm})_{(l,m)\in \Lambda_Z} \given
    \textstyle \sum_{(l,m)\in \Lambda_Z}p_{lm}^2 = 1,\ p_{lm} \geq 0}.
\end{equation*}
We consider the following prior distribution $\Pi$ on $\mathbb S^2(\mathbb
R)$. Let $P_Z$ be a distribution over $\mathbb R^+$ and draw $Z \sim P_Z$. Given
$Z$, draw $\mathbf p$ from a distribution $G(\cdot \mid Z)$ over the simplex
$\Delta_Z$. Independently of $\mathbf p$, draw
$\zeta = (\zeta_{lm})_{(l,m)\in \Lambda_Z}$ from a distribution
$P_{\zeta}(\cdot \mid Z)$ over $[0,2\pi]^{|\Lambda_Z|}$ and set
\begin{equation*}
  \psi := \sum_{(l,m)\in \Lambda_Z} p_{lm} e^{i\zeta_{lm}}\, \varphi_{lm}.
\end{equation*}
Note that $(\varphi_{lm})$ is orthonormal, thus
$\|\psi\|_2^2 = \sum_{(l,m)\in \Lambda_Z}p_{lm}^2 = 1$ almost-surely, that is
$\psi \in \mathbb S^2(\mathbb R)$ almost-surely.

\subsection{Estimation of mixed states}
\label{sec:estim-mixed-stat}

The set of quantum states is a convex set. According to the Hilbert-Schmidt
theorem on the canonical decomposition for compact self-adjoint operators, for
every quantum state $\rho$ there exists an orthonormal set $(\psi_n)_{n=1}^N$ in
$L^2(\mathbb R)$ (finite or infinite, in the latter case $N=\infty$), and
$\alpha_n > 0$ such that
\begin{equation*}
  \rho = \sum_{n=1}^N \alpha_n \rho_{\psi_n},\ \mathrm{and}\qquad
  \mathrm{Tr}\rho = \sum_{n=1}^N\alpha_n = 1.
\end{equation*}
The $(\alpha_n)_{n=1}^N$ are the non-zero eigenvalues of $\rho$ and
$(\rho_{\psi_n})_{n=1}^N$ projection operators onto $(\psi_n)_{n=1}^N$. Thus
every mixed state is a convex linear combination of pure states. In particular,
for any state $\rho$ we have
\begin{equation*}
  W_{\rho}(x,\omega) = \sum_{n=1}^N \alpha_n W_{\psi_n}(x,\omega),
\end{equation*}
making relatively straightforward the extension of priors over pure states onto
priors over general states. In other words, a prior distribution over general
states can be constructed as a mixture of pure states by a random probability
measure.

\section{Simulations examples}
\label{sec:simul-real-data}

\subsection{Simulation procedure}
\label{sec:simulation-procedure}

We test the Gamma process mixtures of coherent states on two examples of quantum
states, corresponding to the Schrödinger cat and $2$-photons states, that are
respectively described by the wave functions
\begin{gather*}
  \psi_{\mathrm{cat}}^{x_0}(x) :=
  \frac{\exp(-\pi(x-x_0)^2) + \exp(-\pi(x+x_0)^2)}
  {2^{1/4}\sqrt{1 + \exp(-2\pi x_0^2)}},\\
  \psi_2(x) := 2^{-1/4}(4\pi x^2 - 1)\exp(-\pi x^2).
\end{gather*}
Using \cref{eq:11,eq:18}, it is seen that the conditional density on
$\theta \in [0,\pi]$ corresponding to the measurement of $\mathbf X_{\theta}$ on
the systems in states $\psi_{\mathrm{cat}}^{x_0}$ and $\psi_2$ are respectively
given by
\begin{multline*}
  p_{\mathrm{cat}}^{x_0}(x \mid \theta)
  \propto
  \sqrt{2}e^{-2\pi(x - x_0\cos\theta)^2}\\
  + \sqrt{2}e^{-2\pi(x + x_0\cos\theta)^{2}}
  + 2e^{-2\pi x_0^2}\frac{\sqrt{2}e^{-2\pi x^2}\cos(4\pi x x_0 \sin\theta)}
  {e^{-2\pi x_0^2 \sin^2\theta}},
\end{multline*}
and,
\begin{equation*}
  p_2(x \mid \theta) = 2^{-1/2}(4\pi x^2 - 1)^2e^{-2\pi x^2}.
\end{equation*}

Note that $p_{\mathrm{cat}}^{x_0}(\cdot \mid \theta)$ is not a mixture density,
since one term can take negative values. Conditionally on $\theta$ drawn
uniformly on $[0,\pi]$, we simulate $n=2000$ observations from the Schrödinger
cat state with $x_0 = 2$ using $p_{\mathrm{cat}}^{x_0}(\cdot \mid \theta)$ and
the rejection sampling algorithm with candidate distribution
$\frac 12 \mathcal{N}(-x_0\cos\theta, 1/(4\pi)) + \frac 12
\mathcal{N}(x_0\cos\theta, 1/(4\pi))$. Similarly, we simulate $n=2000$
observations from the $2$-photons state using the rejection sampling algorithm
with a Laplace candidate distribution. A Gaussian noise is added to observations
according to \cref{eq:2}, where we choose $\eta = 0.95$, a reasonable efficiency
the physicists say.

\subsection{Simulation results}
\label{sec:simulation-results}

We use the algorithm of \citet{NauletBarat2015} for simulating samples from
posterior distributions of Gamma process mixtures. The base measure $\alpha$ on
$\mathbb R^2 \times [0,2\pi]$ of the mixing Gamma process is taken as the
independent product of a normal distribution on $\mathbb R^2$ with covariance
matrix $\mathrm{diag}(1/2,1/2)$ and the uniform distribution on $[0,2\pi]$.

We ran $3000$ iterations of the algorithm with $p=50$ particles, leading to an
acceptance ratio of approximately $60\%$ for the particle moves and the both
datasets. All random-walk Metropolis-Hastings steps are Gaussians, with
amplitudes chosen to achieve approximately $25\%$ acceptance rates. All the
statistics were computed using only the $2000$ last samples provided by the
algorithm.

\begin{figure}[!htb]
  \centering
  \includegraphics[width=\linewidth]{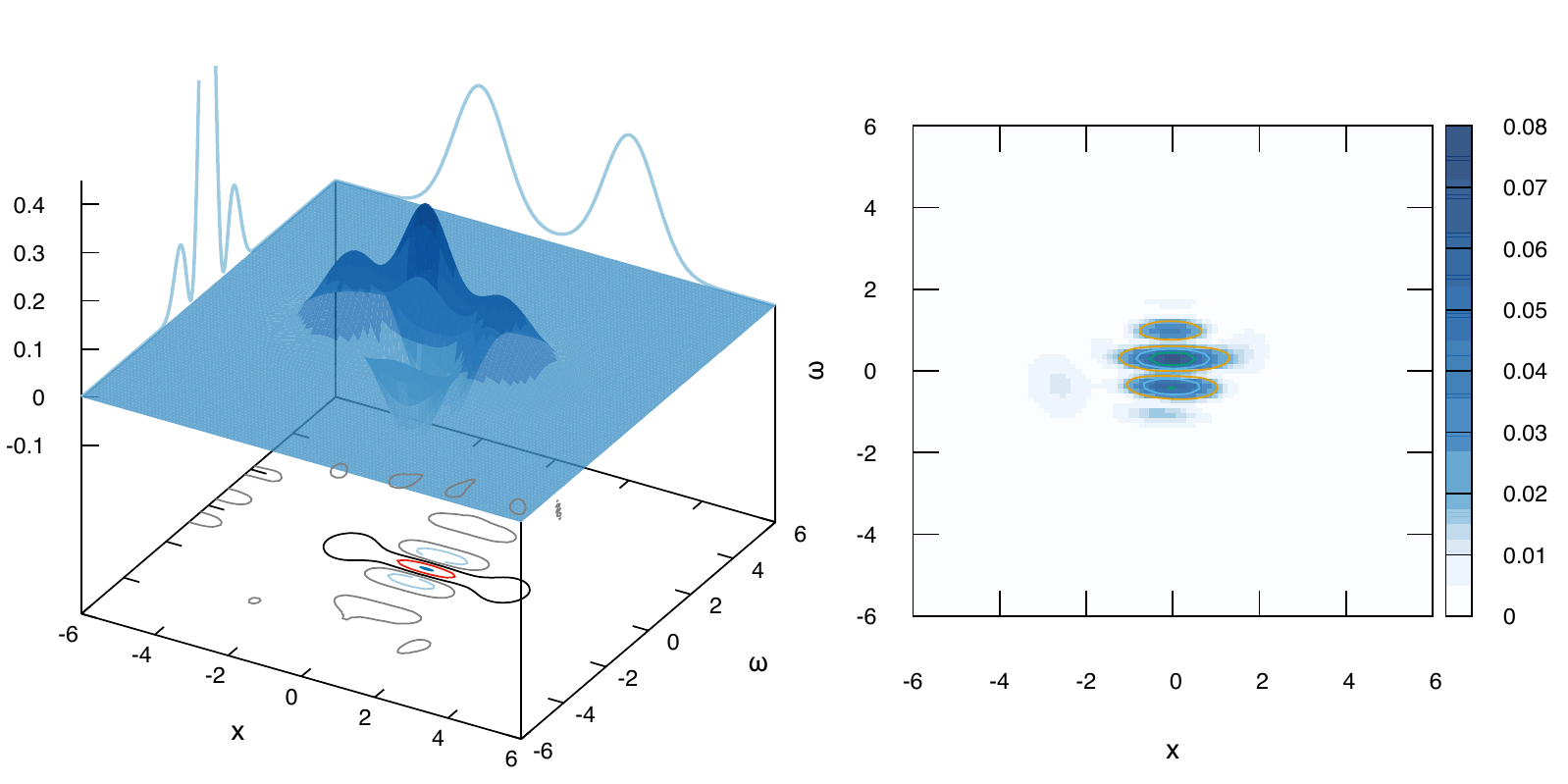}
  \caption{Left: Average of Wigner distribution samples from the posterior
    distribution of the mixture of coherent states prior given $2000$ quantum
    homodyne tomography observations simulated from a Schrödinger cat
    state. Right: View map of the absolute value of the difference between the
    posterior mean estimate of the Wigner distribution and the true Wigner
    distribution.}
  \label{fig:wigner-schrod}
\end{figure}

\begin{figure}[!htb]
  \centering
  \includegraphics[width=\linewidth]{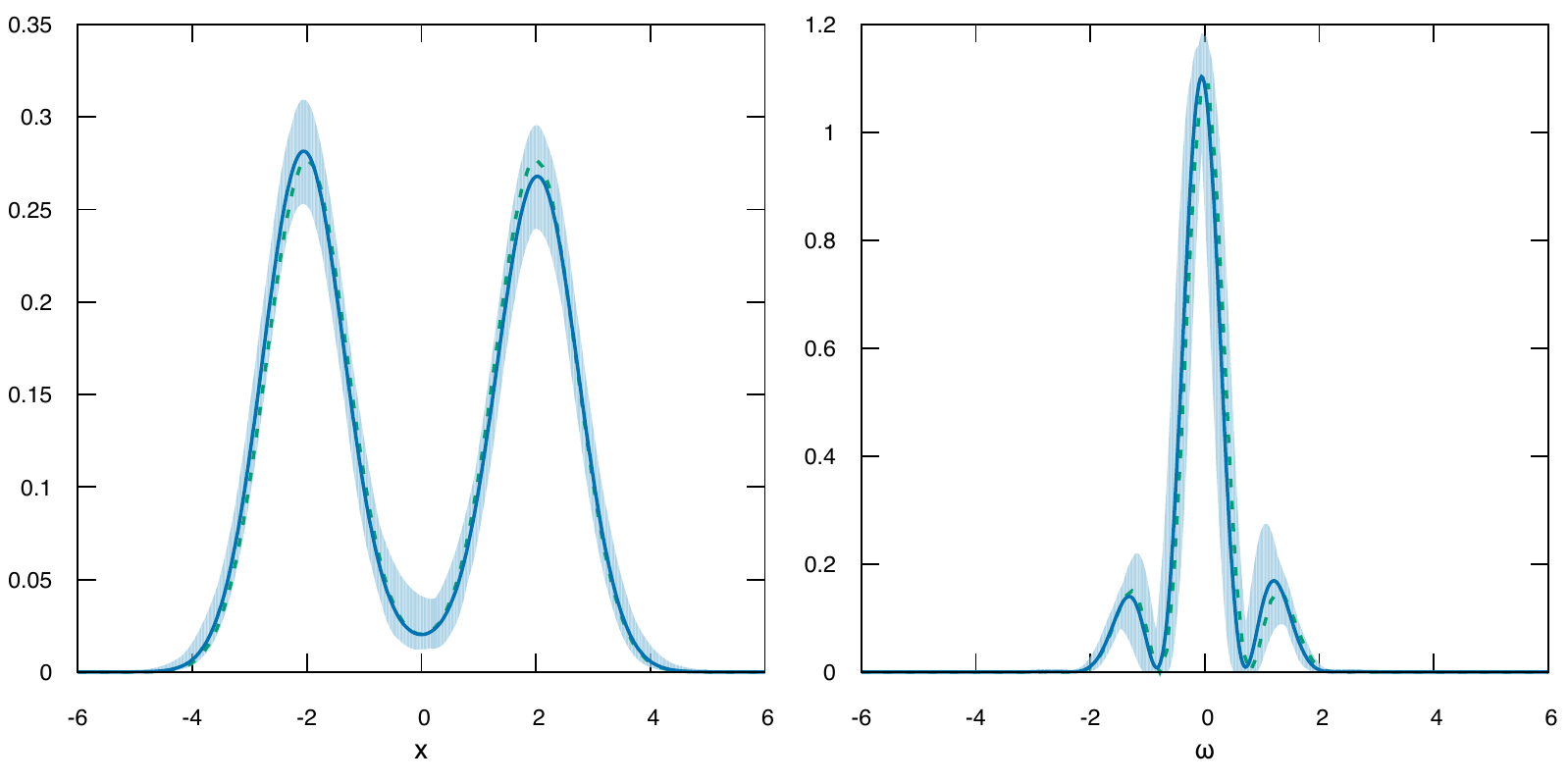}
  \caption{Marginals of the Wigner distribution samples from the posterior
    distribution of the mixture of coherent states prior given $2000$ quantum
    homodyne tomography observations simulated from a Schrödinger cat state. In
    straight line the posterior mean estimate, whereas the dashed lines
    corresponds to the true marginals. The 95\% credible intervals for the
    $\sup$-norm distance are drawn in shading.}
  \label{fig:wigner-schrod-1D}
\end{figure}

\begin{figure}[!htb]
  \centering
  \includegraphics[width=\linewidth]{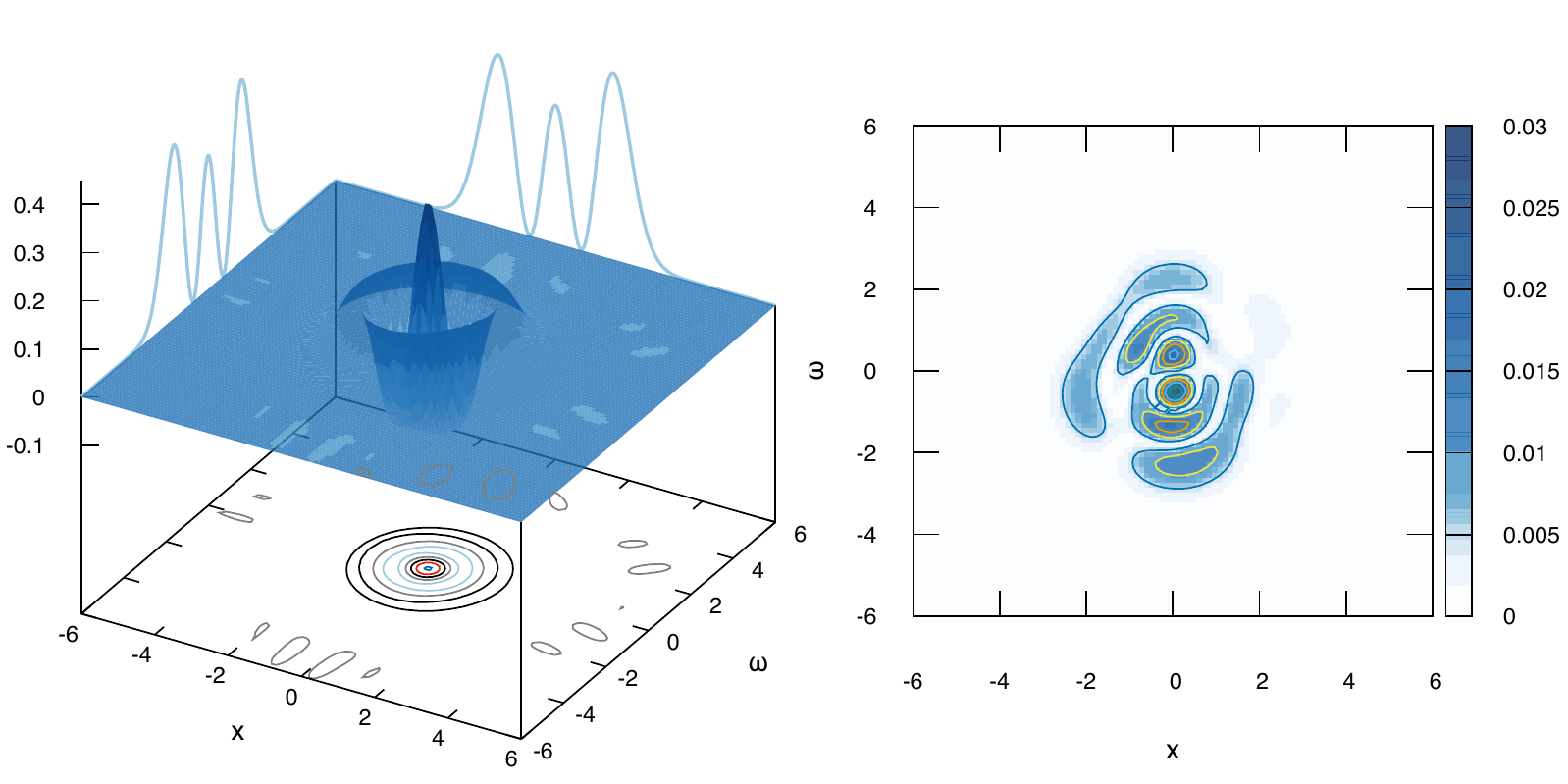}
  \caption{Left: Average of Wigner distribution samples from the posterior
    distribution of the mixture of coherent states prior given $2000$ quantum
    homodyne tomography observations simulated from a $2$-photons state. Right:
    View map of the absolute value of the difference between the posterior mean
    estimate of the Wigner distribution and the true Wigner distribution.}
  \label{fig:wigner-five}
\end{figure}

\begin{figure}[!htb]
  \centering
  \includegraphics[width=\linewidth]{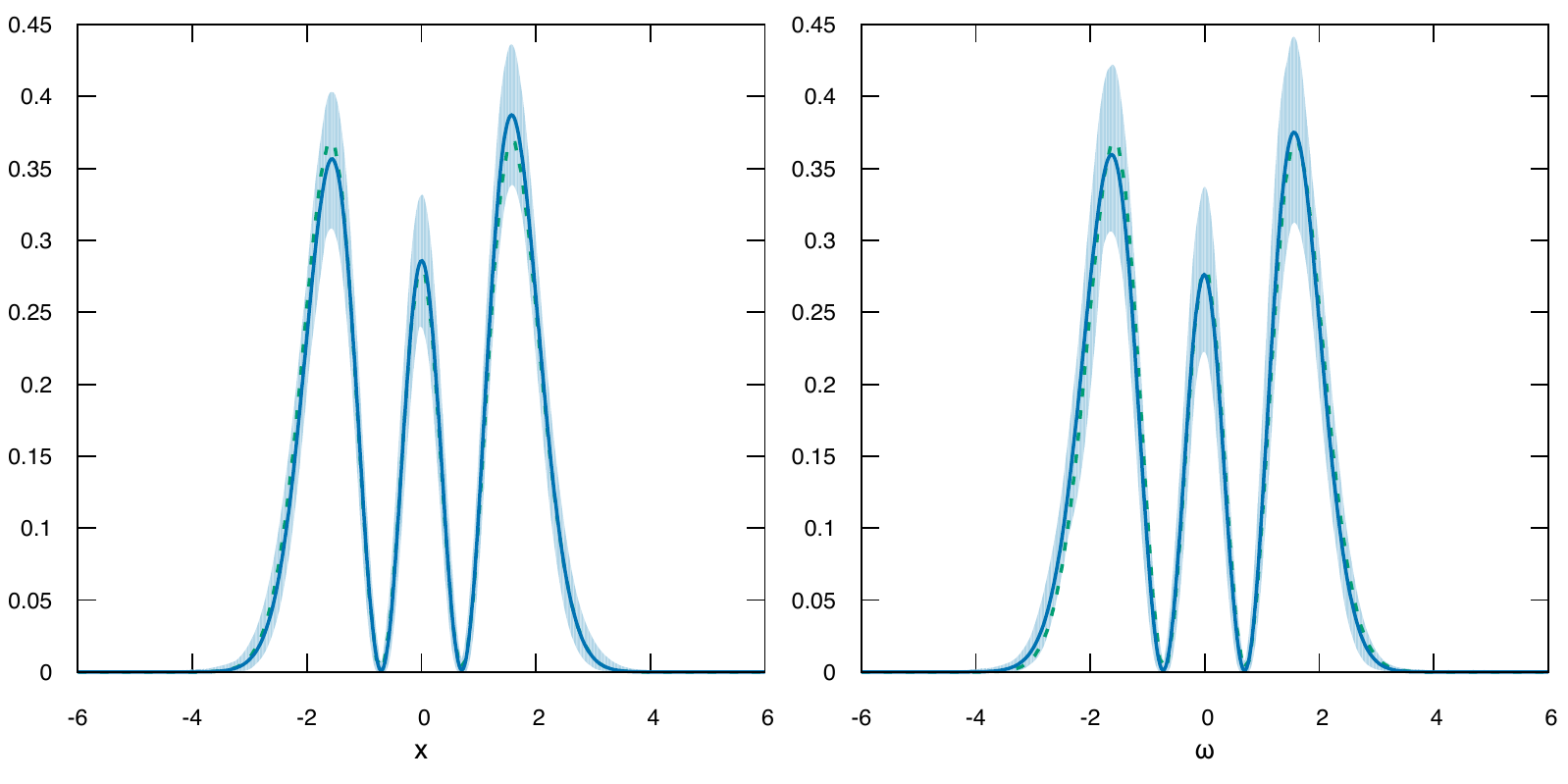}
  \caption{Marginals of the Wigner distribution samples from the posterior
    distribution of the mixture of coherent states prior given $2000$ quantum
    homodyne tomography observations simulated from a $2$-photons state. In
    straight line the posterior mean estimate, whereas the dashed lines
    corresponds to the true marginals. The 95\% credible intervals for the
    $\sup$-norm distance are drawn in shading.}
  \label{fig:wigner-five-1D}
\end{figure}

\Cref{fig:wigner-schrod,fig:wigner-five} represent the average of posterior
samples of the Wigner distribution for the Schrödinger cat state, and the
$2$-photons state, respectively. Because it is hard to distinguish between the
posterior mean estimator and the true Wigner distribution, we added to the
figures a view map of the absolute value of the difference between the evaluated
posterior mean and the true Wigner distribution.

\Cref{fig:wigner-schrod-1D,fig:wigner-five-1D}
show the marginals of the posterior mean estimates of Wigner distributions for
our two examples. We represented the true marginals in dashed lines, as well as
the posterior credible bands provided by the algorithm, which we computed by
retaining the $95\%$ samples with the smaller $\sup$-norm distance from the
posterior mean estimator of the marginals.

Compared to other classical methods in this area, our estimate is non linear,
preventing easy computations. To our knowledge, however, none of the current
approaches can preserve the physical properties of the true Wigner function (non
negativity of marginal distributions, bounds) whereas our approach does
guarantee preservation of all physical properties.

\section{Rates of contraction for random series priors}
\label{sec:rates-contr-rand}

In this section, we establish posterior convergence rates in the quantum
homodyne tomography problem, for estimating pure states. Unfortunately, to get
such result we need a fine control of the $L^2(\mathbb R)$ norm of random
functions drawn from the prior distribution, which remains challenging for
mixtures of coherent states. However, dealing with Wilson bases, the control of
the $L^2(\mathbb R)$ norm is straightforward and we are able to obtain posterior
concentration rates.

\subsection{Preliminaries on function spaces}
\label{sec:prel-funct-spac}

To establish posterior concentration rates, we describe suitable classes of
functions that can be well approximated by partial sums of Wilson bases
elements; these functional classes are called \textit{ultra-modulation
  spaces}. To this aim, we need the following ingredients: the short-time
Fourier transform (STFT), a class of windows and a class of weights. For a
non-zero window function $g\in L^2(\mathbb R)$, the short-time Fourier transform
of a function $f\in L^2(\mathbb R)$ with respect to the window $g$ is given by
\begin{equation}
  \label{eq:8}
  V_gf(x,\omega) := \langle f,\, M_{\omega}T_xg \rangle = \int_{\mathbb R}f(t)
  \overline{g(t-x)}e^{-2\pi i \omega t}\, dt,\quad (x,\omega)\in \mathbb R^2.
\end{equation}

We also need a class of analyzing windows $g$ with sufficiently good
time-frequency localization properties. Following,
\citet{Cordero2007a,CorderoPilipovicRodinoEtAl2005,GroechenigZimmermann2004}, we
use the \textit{Gelfand-Shilov} space $\mathcal{S}_{1}^{1}(\mathbb R)$. For any
$d\geq 1$, a function $f : \mathbb R^d \rightarrow \mathbb C$ belongs to the
Gelfand-Shilov space $\mathcal{S}_{1}^{1}(\mathbb R^d)$ if
$f \in \mathcal{C}^{\infty}(\mathbb R^d)$ and there exist real constants $h > 0$
and $k > 0$ such that
\begin{equation*}
  \sup_{x\in \mathbb R^d}|f(x) e^{h\|x\|}| <+\infty,\qquad
  \sup_{\omega\in \mathbb R^d}|\widehat{f}(\omega) e^{k\|\omega\|}| < +\infty.
\end{equation*}

Next, for $\beta > 0$, $g \in \mathcal{S}_1^1(\mathbb R)$, and
$r\in [0,1)$, we consider the exponential weights on $\mathbb R^{2}$ defined by
$x\mapsto \exp(\beta\|x\|^r)$, and we introduce the class of wave-functions
\begin{equation}
  \label{eq:9}
  \mathcal{C}_g(\beta,r,L)
  := \Set*{\psi \in \mathbb S^2(\mathbb R) \given \int_{\mathbb R^2}
    |V_g\psi(z)|\,\exp(\beta \|z\|^r)dz \leq L}.
\end{equation}
The class $\mathcal{C}_g(\beta,r,L)$ is reminiscent to \textit{modulation
  spaces} \citep{Groechenig,Groechenig2006}. Note that it would be interesting
to consider $\mathcal{C}_g(\beta,r,L)$ for $r \geq 1$, since most quantum states
should fall in these classes. There is, however, at least two limitations for
considering $r \geq 1$. First, we use repeatedly in the proofs that
$\exp(\beta \|x+y\|^r)\leq \exp(\beta \|x\|^r)\exp(\beta \|y\|^r)$ for
$r \leq 1$, which is no longer true when $r > 1$. The previous limitation is
indeed not the more serious concerns, since for $r > 1$ we could use that
$\exp(\beta \|x+y\|^r)\leq \exp(2^{r-1}\beta \|x\|^r)\exp(2^{r-1}\beta
\|y\|^r)$. The more serious problem is that, to our knowledge, there is no
Wilson base for $L^2(\mathbb R)$ whose elements fall into
$\mathcal{C}_g(\beta,r,L)$ for $r > 1$ and $\beta > 0$, $L > 0$. The case $r=1$
is more delicate since it depends on the value of $\beta$. For sufficiently
small $\beta > 0$, the results proved in this paper for $r < 1$ should also hold
for $r=1$.

Let also notice that, there is a fundamental limit on the growth of the weights
in the definition of $\mathcal{C}_g(\beta,r,L)$, imposed by Hardy's theorem. If
$r=2$ and $\beta > \pi/2$, the the corresponding classes of smoothness
$\mathcal{C}_g(\beta,r,L)$ are trivial for any $L > 0$
\citep{GroechenigZimmermann2001}.

A critical point regarding the class $\mathcal{C}_g(\beta,r,L)$ is the
dependence on $g$ in the definition. We truly want that for two different
windows $g_0$ and $g_1$ the corresponding smoothness are the same. Fortunately,
we have the following theorem, proved in \cref{sec:proof-thmindepclass}.

\begin{theorem}
  \label{thm:4}
  Let $g,g_0 \in \mathcal{S}_1^1(\mathbb R)$. For all $\beta,L > 0$ and all
  $0 \leq r < 1$ there is a constant $C > 0$, depending only on $g,g_0$, such
  that embedding
  $\mathcal{C}_g(\beta,r,L) \subseteq \mathcal{C}_{g_0}(\beta,r, CL)$ holds.
\end{theorem}

The STFT and the Wigner transform both aim at having a time-frequency
representation of functions in $L^2(\mathbb R)$, and are deeply linked to each
other. However, contrarily to the Wigner transform, the STFT has the advantage
of being a linear operator, which is one reason why we prefer to state the class
$\mathcal{C}_g(\beta,r,L)$ in term of the STFT instead of the Wigner transform.

\subsection{Assumptions and results}
\label{sec:assumptions}

Before stating the main result of this paper, we need some further assumptions
on the random Wilson base series prior, which we state now. To this aim, we need
the following definition of the \textit{weighted} simplex
$\Delta_Z^w(\beta,r, M)$. For a constant $M > 0$, $\beta > 0$ and $r \in [0,1)$
let
\begin{equation*}
  \Delta_{Z}^w(\beta,r,M) := \Set*{\mathbf p \in \Delta_Z \given
    \textstyle
    \sum_{(l,m)\in \Lambda_Z}p_{lm}\exp\left(\beta (l^2 + m^2)^{r/2}\right) <
    M }.
\end{equation*}
Then, in the sequel, we assume that
\begin{itemize}
  \item There is a constant $a_0 > 0$ such that
  for any sequence $(x_{lm})_{(l,m)\in \Lambda_Z}\in [0,2\pi]^{|\Lambda_Z|}$,
  \begin{equation*}
    P_{\zeta}\left(\textstyle \sum_{(l,m)\in \Lambda_Z}|\zeta_{lm} - x_{lm}|^2
      \leq t \mid Z\right) \gtrsim \exp\left(- a_0 Z^2 \log t^{-1}\right),
    \quad \forall t \in (0,1).
  \end{equation*}
  \item $P_Z(Z < +\infty) = 1$ and there are constants $a_1,a_2 > 0$ and
  $b_1 > 2 + r$, such that for all $k$ positive integer large enough
  \begin{equation*}
    P_Z(Z = k) \gtrsim \exp(-a_1k^{b_1}),\qquad
    P_Z(Z > k) \lesssim \exp(-a_2k^{b_1}).
  \end{equation*}
  \item For any constant $C >0$ and any sequence
  $\mathbf q \in \Delta_Z^w(\beta,r,C)$, there is a constant $a_{3} > 0$ such
  that the distribution $G(\cdot \mid Z)$ satisfy,
  \begin{equation*}
    G\left(
      \textstyle\sum_{(l,m)\in \Lambda_Z}|p_{lm} - q_{lm}|^2 \leq t \mid
      Z\right) \gtrsim \exp\left(- a_3 Z^{b_1-r} \log t^{-1} \right),\
    \forall t\in (0,1).
  \end{equation*}
  We further assume that there exist constants $a_4 \geq 0$, $a_5,c_0 > 0$, and
  $b_5 > b_1/r$ such that for $x > 0$ large enough
  \begin{equation*}
    G\left( \mathbf p \notin \Delta_Z^w(\beta,r, c_0x^{a_4})
      \mid Z \leq x^{1/r} \right) \lesssim \exp\left( - a_5 x^{b_5}\right).
  \end{equation*}
\end{itemize}

It is not clear whether or not we can find a distribution $G$ for which the
above conditions are satisfied simultaneously for all $(\beta,r,L)$, eventually
with constants $a_3,a_4,a_5,b_5$ depending on $(\beta,r,L)$. If such
distribution exists, then the rates stated below are easily seen to be adaptive
on $(\beta,r,L)$. In \cref{sec:block-spike-slab}, we show that for a given
$(\beta,r,L)$ it is easy to construct a distribution $G$ that satisfies the
above conditions, with $a_4=2/r$. However, we believe that the proof for
adaptive rates must follow a different path, still to be found.

Under the hypothesis above, we will dedicate the rest of the paper to prove the
following theorem.

\begin{theorem}
  \label{thm:5}
  Let $\beta,L > 0$ and $r \in (0,1)$. Let $\Pi$ be the random Wilson series
  prior satisfying the assumptions above, and
  $(Y_1,\theta_1),\dots,(Y_n,\theta_n)$ be observations coming from the
  statistical model described by \cref{eq:2}, with $0 < \eta < 1$ and
  $\gamma >0$ defined in \cref{eq:12}. Then for any
  $\psi_0 \in \mathcal{C}_g(\beta,r,L)$, there is $M > 0$ such that
  \begin{gather*}
    P_{\psi_0}^{\eta,n} \Pi( \|\psi - \psi_0\|_2 \geq M\epsilon_n \mid
    (Y_1,\theta_1),\dots, (Y_n,\theta_n)) \rightarrow 0,\\
    \epsilon_n^2 = (\log n)^{2a_4}\exp\left\{ - \beta \left( \frac{\log n}{2
          \gamma} \right)^{r/2} + O(1) \right\}.
  \end{gather*}
\end{theorem}

Note that the same result holds with $\|\psi - \psi_0\|_2$ replaced with
$\|W_{\psi} - W_{\psi_0}\|_2$, because the Wigner transform is isometric from
$L^2(\mathbb R)$ onto $L^2(\mathbb R^2)$; see for instance
\citet[proposition~4.3.2]{Groechenig}.

The rates of contraction are relatively slow, a fact that is also pointed out in
\citet{ButuceaGutaArtilesEtAl2007}. Indeed, the rates are faster than
$(\log n)^{-a}$ but slower than $n^{-a}$, for all $a > 0$. The reason for such
bad rates of convergence is to be found in the deconvolution of the Gaussian
noise. If one does not carry about deconvoluting the noise, then all the steps
in the proof of \cref{thm:5} can be mimicked to get weaker a result. In
particular, we infer from the results of the paper that the posterior
distribution should contracts at nearly parametric rates, \textit{i.e.} at rate
$\epsilon_n \approx n^{-1/2}(\log n)^t$ for some $t > 0$, around balls of the
form
\begin{equation}
  \label{eq:16}
  \Set*{\psi \in \mathbb S^2(\mathbb R) \given \int_{\mathbb R^2}
    |\widehat{W}_{\psi}(z) - \widehat{W}_{\psi_0}(z)|^2\,
    \widehat{G}_{\gamma}(\|z\|)^2dz \leq \epsilon_n^2},
\end{equation}
whenever $\psi_0 \in \mathcal{C}_g(\beta,r, L)$ for some $\beta,L > 0$ and
$r \in (0,1)$. Moreover, we've made many restrictive assumptions on the prior
distribution that can be easily released for those interested only in posterior
contraction around balls of the form \eqref{eq:16}.

A natural question regarding the rates found in \cref{thm:5} concerns
optimality. We do not know yet the minimax lower bounds over the class
$\mathcal{C}_g(\beta,r,L)$ for the $L^2$ risk. However,
\citet{ButuceaGutaArtilesEtAl2007, AubryButuceaMeziani2008,
  LouniciMezianiPeyre2015} consider a class $\mathcal{A}(\alpha,r,L)$ that
resembles to $\mathcal{C}_g(\beta,r,L)$. More precisely, they define
\begin{equation*}
  \mathcal{A}(\alpha,r,L) := \Set*{W_{\rho} \given \int
    |\widehat{W}_{\rho}(z)|^2\, \exp(2\alpha \|z\|^r)dz \leq L^2}.
\end{equation*}
Identifying $\rho_{\psi}$ with $\psi$, our \cref{pro:8} state the embedding
$\mathcal{C}_g(\beta,r,L) \subseteq \mathcal{A}(\beta/2,r,L)$. Hence
$\mathcal{C}_g(\beta,r,L)$ is certainly contained in the intersection of a class
$\mathcal{A}(\beta/2,r,L)$ with the set of pure states, and it makes sense to
compare the rates. To our knowledge, the only minimax lower bound for the
quadratic risk known is for the estimation of a state in
$\mathcal{A}(\alpha,r=2,L)$, stated in \citet{LouniciMezianiPeyre2015}. For
$r \in (0,1)$, however, upper bounds for the quadratic risk over
$\mathcal{A}(\beta/2,r,L)$ are established in \citet{AubryButuceaMeziani2008},
and coincide with the rates found here. Therefore, we believe that the rates we
found in this paper are optimal.

Let conclude with a few points that are still challenging at this time. First,
the rates (or even consistency) for the coherent states mixtures priors appears
difficult to establish with the method employed here; the reason comes from the
difficulty to control the norm $\|\widetilde{\psi}\|_2$ when $\widetilde{\psi}$
is a coherent states mixture. Regarding Wilson based priors, we already
discussed the lack of adaptivity, which clearly deserved to be dug in a near
future. Finally, it would be interesting to consider priors based on Gabor
frames expansions, as they are more flexible than Wilson bases, and should be
computationally more efficient than coherent states mixtures. However, Gabor
frames suffer from the same evil that coherent states, namely the expansions are
not unique and it is hard to control from below the $L^2$ norm of random Gabor
expansions.

\section{Example of priors on the simplex}
\label{sec:block-spike-slab}

In this section, we construct a prior on the simplex $\Delta_Z$ that satisfy
the assumptions of \cref{sec:assumptions} for a given $(\beta,r)$. For all
$k \geq 1$, and a constant $M > 0$ to be defined later, we define the sets
\begin{equation*}
  \mathcal{I}_k := \Set*{ (l,m) \in \mathbb N \times \textstyle \frac 12 \mathbb
    Z \given (k-1)M \leq \sqrt{l^2 + m^2} < k M}.
\end{equation*}
We assume without loss of generality that $Z = KM$ for an integer $K > 0$; then
$\Lambda_Z = \cup_{k=1}^K \mathcal{I}_k$. We then construct the distribution
$G(\cdot \mid Z)$ over the simplex $\Delta_Z$ as follows. For $k = 2,\dots,K$,
let $H_k$ be the uniform distribution over
$[0,\sqrt{2}L\exp(-\beta(k^r-1) M^r)]$. Let $\theta_1 := 1$ and for
$k=2,\dots,K$ draw $\theta_k$ from $H_k$ independently. The next step is to
introduce distributions $F_k$ over the $\mathcal{I}_k$-simplex
\begin{equation*}
  \mathcal{S}_k :=
  \textstyle
  \Set*{(\eta_{lm})_{(l,m)\in \mathcal{I}_k}
    \given \sum_{(l,m)\in \mathcal{I}_k} \eta_{lm}^2 = 1,\quad \eta_{lm} \geq
    0},
\end{equation*}
and draw independently sequences
$(\eta_{lm})_{(l,m)\in \mathcal{I}_1}, (\eta_{lm})_{(l,m)\in
  \mathcal{I}_2},\dots, (\eta_{lm})_{(l,m)\in \mathcal{I}_K}$, according to
distributions $F_1,F_2,\dots, F_K$. Finally, the sequence
$\mathbf p = (p_{lm})_{(l,m)\in \Lambda_Z}$ drawn from $G(\cdot \mid Z)$ is
defined to be such that
\begin{equation*}
  p_{lm} := \frac{\eta_{lm}
  \textstyle \theta_k \mathbbm 1\left( (l,m)\in
    \mathcal{I}_k \right).}{\sum_{k=1}^K \theta_k^2}.
\end{equation*}
Now we prove that we can chose reasonably $M > 0$ and the distributions
$F_1,F_2\dots$ to met the assumptions of \cref{sec:assumptions}. The proofs of
the next two propositions are to be found in \cref{sec:proofs-unif-seri}.

\begin{proposition}
  \label{pro:4}
  There is a constant $c_0 > 0$ such that for any $Z \geq 0$ it holds
  $(p_{lm})_{(l,m)\in \Lambda_Z} \in \Delta_Z^w(\beta,r, c_0 Z^2)$ with
  $G(\cdot \mid Z)$ probability one.
\end{proposition}

\begin{proposition}
  \label{pro:9}
  Let $M > 0$ be large enough, $K \geq 0$ integer, and $Z = KM$. Assume that
  there is a constant $c_0 > 0$ and a sequence $(d_k)_{k=1}^K$ such that
  $\sum_{k=1}^Kd_k \leq c_0K$, and for any sequence
  $(e_{lm})_{(l,m)\in \mathcal{S}_k}$ it holds
  $F_k(\sum_{(l,m)\in \mathcal{I}_k}|\eta_{lm} - e_{lm}|^2 \leq t) \gtrsim
  \exp(-d_k K^{b_1 - r -1} \log t^{-1})$. Then there is a constant $a_3 > 0$
  such that
  \begin{equation*}
    G\left(
      \textstyle
      \sum_{(l,m)\in \Lambda_Z}|p_{lm} - q_{lm}|^2 \leq 12t \mid Z\right)
    \gtrsim \exp(-a_3Z^{b_1 - r}\log t^{-1}).
  \end{equation*}
\end{proposition}

In the previous proposition, some conditions are required on $F_1,F_2,\dots$;
these conditions are indeed really mild. For instance, it follows from
\citet[lemma~6.1]{GhosalGhoshVanDerVaart2000} that the conclusion of
\cref{pro:9} is valid if $\eta_{lm} := \sqrt{u_{lm}}$ where
$(u_{lm})_{(l,m)\in \mathcal{I}_k}$ are drawn from Dirichlet distributions with
suitable parameters.

\section{Proof of \texorpdfstring{\cref{thm:5}}{theorem \ref{thm:5}}}
\label{sec:proof-thm:xx}

The proof of \cref{thm:5} follows the classical approach of
\citet{GhosalGhoshVanDerVaart2000,Ghosal2007} for which the prior mass of
Kullback-Leibler type neighborhoods need to be bounded from below and tests
constructed. See details in \cref{sec:bound-post-distr}.

Throughout the document, we let $D_n^{\beta,r} := (\log(n)/\beta)^{1/r}$. Then
we introduce the following events, which we'll use several times in the proof of
posterior contraction rates.
\begin{gather}
  \label{eq:7}
  E_n := \Set*{(y,\theta) \in \mathbb R \times [0,2\pi] \given
    |y| \leq D_n^{\beta,r} },\\
  \label{eq:3}
  \Omega_n := \Set*{((y_{1},\theta_1),\dots,(y_n,\theta_n))
    \given (y_i,\theta_i) \in E_n \quad \forall i=1,\dots,n}.
\end{gather}

\subsection{Prior mass of Kullback-Leibler neighborhoods}
\label{sec:prior-mass-kullback}

We introduce a new variation around the basic lines of
\citet{GhosalGhoshVanDerVaart2000,Ghosal2007}, permitting to slightly weaken the
so-called \textit{Kullback-Leibler} (KL) condition. We show that we can trade
the KL condition for a restricted KL condition; that is prior positivity of the
sets
\begin{equation}
  \label{eq:5}
  B_n(\delta_n) := \Set*{\psi \given \int_{E_n}
    p_{\psi_0}^{\eta} \log
    \frac{p_{\psi_0}^{\eta}}{p_{\psi}^{\eta}} \leq \delta_n^2,\quad
    \int_{E_n}
    p_{\psi_0}^{\eta} \left( \log
      \frac{p_{\psi_0}^{\eta}}{p_{\psi}^{\eta}} \right)^2 \leq \delta_n^2}.
\end{equation}
Although looking trivial, this will ease the proof of our main theorem, since
the prior positivity of $B_n(\delta_n)$ is simpler to prove than the classical
positivity of KL balls of \citet{GhosalGhoshVanDerVaart2000,Ghosal2007}.

\subsubsection{Decay estimates of the true density}
\label{sec:decay-estimates}

It is a classical fact that in Bayesian nonparametrics we often require tails
assumptions on the density of observations to be able to state rates of
convergence. Here, the density of observations is quite complicated, as being
the convolution of a Gaussian noise with the Radon-Wigner transform of
$\psi$. Since the Wigner transform of $\psi$ interpolates $\psi$ and its Fourier
transform, we definitively have to take care about fancy tails assumptions on
the density that could be non compatible with the requirements of a Wigner
transform. Instead, we show that the decay assumptions on the STFT stated in the
definition of $\mathcal{C}_g(\beta,r,L)$ directly translate onto the tails of
the joint density of observations. We have the following theorem, whose proof is
given in \cref{sec:proof-thm1-lem}.
\begin{lemma}
  \label{thm:1}
  For all $\beta,L > 0$ and all $r \in (0,1)$ there is a constant
  $C(\beta,r,\eta) > 0$ such such that
  $P_{\psi}^{\eta}(E_n^c) \leq 2\pi C(\beta,r,\eta)L^2 n^{-2}$ and
  $P_{\psi}^{\eta,n}(\Omega_n^c) \leq 2\pi C(\beta,r,\eta)L^2 n^{-1}$ for all
  $\psi \in \mathcal{C}_g(\beta,r,L)$.
\end{lemma}

\subsubsection{Approximation theory}
\label{sec:approximation-theory}

In order to prove the prior positivity of the sets $B_n(\delta_n)$, we need to
construct a family $\mathcal{M}_n$ of functions in $\mathbb S^2(\mathbb R)$ that
approximate well $\psi_0$ in the $L^2(\mathbb R)$ distance. We will show later
that the sets $B_n(\delta_n)$ contains suitable closed balls around $\psi_0$ in
the norm of $L^2(\mathbb R)$.

In the sequel, we need to relate the parameters $\beta,r,L$ to the decay of the
coefficients $\langle \psi_0,\, \varphi_{lm}\rangle$ of
$\psi_0 \in \mathcal{C}_g(\beta,r,L)$ expressed in the Wilson base. Fortunately,
Wilson bases are unconditional bases for the ultra-modulation spaces, and
$\mathcal{C}_g(\beta,r,L)$ is a subset of the ultra-modulation space
$M^1_{\beta,r}$. It follows the following lemma
\citep[theorem~12.3.1]{Groechenig}.

\begin{lemma}
  \label{lem:9}
  Let $\psi \in \mathcal{C}_g(\beta,r,L)$ for some $\beta,L > 0$ and
  $0 \leq r < 1$. Then there is a constant $0 < C(\beta,r) < +\infty$ such that
  \begin{equation*}
    \sum_{(l,m)\in \Lambda_{\infty}}|\langle \psi,\, \varphi_{lm} \rangle
    |\exp\left(\beta(l^2 + m^2)^{r/2} \right) \leq C(\beta,r)L.
  \end{equation*}
\end{lemma}

Having characterized the decay of Gabor coefficients for those
$\psi \in \mathcal{C}_g(\beta,r,L)$, we are now in position to construct
functions $\psi_Z$ which degree of approximation to
$\psi_0 \in \mathcal{C}_g(\beta,r,L)$ is indexed by the value of $Z$. In view of
\cref{sec:wilson-bases}, $\psi_0$ has the formal decomposition
$\psi_0 = \sum_{l,m}\langle \psi_0,\, \varphi_{lm} \rangle\, \varphi_{lm}$, with
unconditional convergence of the series in $L^2(\mathbb R)$. We define
$\widetilde{\psi}_Z$ such that
\begin{equation*}
  \widetilde{\psi}_Z := \sum_{(l,m)\in \Lambda_Z}\langle \psi_0,\, \varphi_{lm}
  \rangle\, \varphi_{lm}.
\end{equation*}
Since $(\varphi_{lm})$ constitutes an orthonormal base for $L^2(\mathbb R)$,
\cref{lem:9} implies that for any $\beta > 0$ and $r \in (0,1)$,
\begin{align*}
  \|\psi_0 - \widetilde{\psi}_Z\|_2^2
  &= \sum_{(l,m)\notin \Lambda_Z}|\langle
    \psi_0,\, \varphi_{lm} \rangle|^2\\
  \notag
  &\leq \exp(-\beta Z^r) \sum_{l,m} |\langle \psi_0,\, \varphi_{lm} \rangle|
    \exp\left(\beta (l^2 + m^2)^{r/2}\right)\\
  \notag
  &\leq C(\beta,r)L \exp\left(-\beta Z^r \right),
\end{align*}
because on $\Lambda_Z^c$ we have $l^2 + m^2 \geq Z^2$ and
$|\langle \psi_0,\, \varphi_{lm}\rangle| \leq \|\psi_0\|_2\|\varphi_{lm}\|_2 =
1$. Note that $\widetilde{\psi}_Z$ is not necessarily in
$\mathbb S^2(\mathbb R)$, that is in general $\|\widetilde{\psi}_Z\|_2 \neq 1$,
whence it is not a proper wave-function. We now trade $\widetilde{\psi}_Z$ for a
version $\psi_Z$ with $\|\psi_Z\|_2 = 1$, keeping the same order of
approximation. Indeed, let
$\psi_Z := \widetilde{\psi}_Z / \|\widetilde{\psi}_Z\|_2$, then since
$\|\psi_0\|_2 = 1$,
\begin{align}
  \|\psi_Z - \psi_0\|_2
  \notag
  &\leq \|\psi_Z - \widetilde{\psi}_Z\|_2 + \|\widetilde{\psi}_Z - \psi_0\|_2\\
  \notag
  &\leq \|\widetilde{\psi}_Z\|\left|1 -
    \frac{1}{\|\widetilde{\psi}_Z\|_2}\right|
    + \|\widetilde{\psi}_Z - \psi_0\|_2 \leq 2\|\widetilde{\psi}_Z -
    \psi_0\|_2\\
  \label{eq:20}
  &\leq 2\sqrt{C(\beta,r)L} \exp\left(-\frac{\beta Z^r}{2} \right).
\end{align}

\subsubsection{A lower bound on \texorpdfstring{$\Pi(B_n(\delta_n))$}{the prior
    mass of KL neighborhoods}}
\label{sec:kullb-leibl-prop-old}

The proof of the lemmas and theorem of this section are to be found in
\cref{sec:proofs-regard-appr,sec:proof-lower-bound}. To prove the
Kullback-Leibler condition, we first construct a suitable set
$\mathcal{M}_n \subset B_n(\delta_n)$, and we'll lower bound
$\Pi(B_n(\delta_n))\geq \Pi(\mathcal{M}_n)$. Let $\psi_Z$ be the function
constructed in \cref{sec:approximation-theory} and
$c_{lm} := \langle \psi_Z,\, \varphi_{lm} \rangle$, so that
$\psi_Z = \sum_{(l,m)\in \Lambda_Z}c_{lm} \varphi_{lm}$. Then, we define the set
$\mathcal{M}_n \equiv \mathcal{M}_n(Z,U)$ as follows, and we'll prove that $Z,U$
can be chosen so that $\mathcal{M}_n(Z,U) \subset B_n(\delta_n)$.
\begin{equation}
  \label{eq:13}
  \mathcal{M}_n(Z,U) := \Set*{ \psi \in \mathbb S^2(\mathbb R)
    \given
    \begin{array}{l}
      \psi = \sum_{(l,m)\in \Lambda_Z}p_{lm}e^{i\zeta_{lm}}\, \varphi_{lm},\\
      \sum_{(l,m)\in \Lambda_Z}|p_{lm} - |c_{lm}| |^2 \leq U^2\\
      \sum_{(l,m)\in \Lambda_Z}|\zeta_{lm} - \arg c_{lm} |^2 \leq U^2
    \end{array}
  }.
\end{equation}

\begin{lemma}
  \label{lem:2}
  For all $\psi \in \mathcal{M}_n(Z,U)$, it holds with the constant
  $C(\beta,r)$ of \cref{lem:9},
  \begin{equation*}
    \|\psi - \psi_0\|_2 \leq
    2U + 2\sqrt{C(\beta,r,g) L}\exp\left(- \frac{\beta Z^r}{2} \right).
  \end{equation*}
\end{lemma}

The fact that $\mathcal{M}_n(Z,U)$ is included into a suitable $L^2(\mathbb R)$
ball around $\psi_0$ is not enough to prove the inclusion
$\mathcal{M}_n(Z,U) \subset B_n(\delta_n)$. The next lemma states sufficient
conditions for which the inclusion $\mathcal{M}_n(Z,U) \subset B_n(\delta_n)$
actually holds true.

\begin{lemma}
  \label{lem:4}
  There are constants $0 < C_1,C_2 < \infty$ depending only on
  $\gamma,\beta,r,A,B,L$ such that if $U \leq C_1(\log n)^{-4/r}\delta_n^2$ and
  $Z\geq C_2(\log \delta_n^{-1})^{1/r}$, then for $n$ large enough
  $\mathcal{M}_n(Z,U) \subset B_n(\delta_n)$ for every
  $\delta_n^2 \geq 4\sqrt{2\pi C(\beta,r,\eta)}L n^{-1}$, where
  $C(\beta,r,\eta)$ is the constant of \cref{thm:1}.
\end{lemma}

Now that we have shown that $\mathcal{M}_n(Z,U) \subseteq B_n(\delta_n)$ for
suitable choice of $Z$ and $U$, it is clear that the prior mass of
$B_n(\delta_n)$ is lower bounded by the prior mass of $\mathcal{M}_n(Z,U)$, the
one is relatively easy to compute. This statement is made formal in the next
theorem.
\begin{theorem}
  \label{thm:3}
  Let $\psi_0 \in \mathcal{C}_g(\beta,r,L)$, and $b_1 > 2 + r$. Then there is a
  constant $C > 0$ such that for $n\delta_n^2 = C(\log n)^{b_1/r}$ it holds
  $\Pi(B_n(\delta_n)) \gtrsim \exp(-n\delta_n^2)$ for $n$ large enough.
\end{theorem}

\subsection{Construction of tests}
\label{sec:constr-tests}

The approach for constructing tests is reminiscent to
\citet{KnapikSalomond2014}, where authors provide a general setup to establish
posterior contraction rates in nonparametric inverse problems. We define the
following sieve. For positive constants $c, h$ to be determined later, and the
constant $a_4 > 0$ of the assumptions
\begin{equation*}
  \mathcal{F}_n
  :=
  \Set*{ \psi \in \mathbb S^2(\mathbb R)
    \given
    \begin{array}{l}
      \psi = \sum_{(l,m)\in \Lambda_Z}
      p_{lm}e^{i\zeta_{lm}}\, \varphi_{lm},\quad
      0 \leq Z \leq h (\log n)^{1/r},\\
      \mathbf p \in \Delta_Z^w(\beta,r, c (\log n)^{a_4})
    \end{array}
  }.
\end{equation*}

Then, we construct test functions with rapidly decreasing type I and type II
errors, for testing the hypothesis $H_0 : \psi = \psi_0$ against the alternative
$H_1 : \psi \in U_n \cap \mathcal{F}_n$, with
$U_n := \Set{\psi \in \mathbb S^2(\mathbb R) \given \|\psi - \psi_0 \|_2\geq
  \epsilon_n}$, for a sequence $(\epsilon_n)_{n\geq 0}$ to be determined
later. To this aim, we need the following series of propositions about
$\mathcal{F}_n$, which are proved in \cref{sec:proofs-sieve-constr}.

\begin{proposition}
  \label{pro:5}
  Let $n\delta_n^2 = C(\log n)^{b_1/r}$ for some constant $C > 0$. Then
  $\Pi(\mathcal{F}_n^c) \lesssim \exp(-6n\delta_n^2)$ whenever
  $h > (6C/a_2)^{1/b_1}$ and $c > 0$ large enough.
\end{proposition}

\begin{proposition}
  \label{pro:6}
  Let $b_1 > 2 + r$ and assume that $n\delta_n^2 = C(\log n)^{b_1/r}$ for some
  constant $C > 0$. Then
  $N(\sqrt{2}\delta_n^2, \mathcal{F}_n, \|\cdot\|_2)\exp(-6n\delta_n^2) = o(1)$.
\end{proposition}

\begin{proposition}
  \label{pro:7}
  There is a constant $M > 0$, depending only on $\varphi$ and $\eta$, such that
  for all $\psi \in \mathcal{F}_n$ it holds
  $\|p_{\psi}^{\eta}\|_{\infty} \leq M h^2(\log n)^{2/r}$.
\end{proposition}

\begin{proposition}
  \label{pro:1}
  For all $\beta > 0$ and $r \in (0,1)$ there is a constant $R > 0$ such that
  for any $u > 0$ it holds
  $\sup_{\psi \in \mathcal{F}_n}\int_{\Set{\|z\|>
      u}}|\widehat{W}_{\psi}(z)|^{2}\,dz \leq R (\log n)^{2a_4}\exp(-\beta
  u^r)$.
\end{proposition}

The first step in the tests construction consists on bounding, both from below
and from above, the Hellinger distance
$H^2(P_{\psi}^{\eta},P_{\psi_{0}}^{\eta})$ by a multiple constant of
$\|\psi - \psi_0\|_2$, at least for those $\psi_0 \in \mathcal{C}_g(\beta,r,L)$
and those $\psi \in \mathcal{F}_n$. To this aim, we need to estimate the decay
of $\widehat{W}_{\psi_0}$, stated in the next proposition. The remaining proofs
for this section can be found in
\cref{sec:proofs-norm-equiv,sec:constr-glob-test}.

\begin{proposition}
  \label{pro:8}
  Let $\psi \in \mathcal{C}_g(\beta,r,L)$ for some $\beta,L > 0$ and
  $r\in (0,1)$. Then
  \begin{equation*}
    \int_{\mathbb R^2}|\widehat{W}_{\psi}(z)|^2 \exp(\beta\|z\|^r) \,dz \leq
    L^2.
  \end{equation*}
\end{proposition}

The practical \cref{pro:8} allows to upper bound $\| \psi - \psi_0\|_2$ by
$H(P_{\psi}^{\eta},P_{\psi_0}^{\eta})$, provided $\psi$ and $\psi_0$ are
sufficiently separated from each other.

\begin{lemma}
  \label{lem:3}
  Let $\beta,L > 0$, $r\in (0,1)$, $C_0 := \|p_{\psi_0}^{\eta}\|_{\infty}$,
  $M,R > 0$ be the constants of \cref{pro:7,pro:1}, and assume $n$ large
  enough. Then for all $u > 0$, all $\psi \in \mathcal{F}_n$ and all
  $\psi_0 \in \mathcal{C}_g(\beta,r,L)$ such that
  $\|\psi - \psi_0\|_2^2 \geq 8R (\log n)^{2a4}\exp(-\beta u^r)$, it holds
  \begin{equation*}
    \sqrt{2}H^2(P_{\psi}^{\eta},P_{\psi_0}^{\eta})
    \leq \|\psi - \psi_0\|_2 \leq 2\sqrt{C_0 + Mh^2 (\log n)^{2/r}}e^{\gamma
      u^2} H(P_{\psi}^{\eta},P_{\psi_0}^{\eta}).
  \end{equation*}
\end{lemma}

From the last lemma, we are in position to construct test functions with rapidly
decreasing type I and type II error for testing
$H_0 : \psi = \psi_0 \in \mathcal{C}_g(\beta,r,L)$ against
$H_1 : \|\psi - \psi_1\|_2 \leq \sqrt{2}\delta_n^2$ for any
$\psi_1 \in \mathcal{F}_n$ such that $\|\psi_1 - \psi_0\|_2 \geq \epsilon_n^2$,
with
\begin{equation}
  \label{eq:6}
  \delta_n^2 := \frac{\epsilon_n^2\exp(-2\gamma u_n^2)}{48[C_0 + Mh^2(\log
    n)^{2/r})]} ,\qquad \epsilon_n^2 := 8R (\log n)^{2a4}\exp(-\beta u_n^r),
\end{equation}
where $(u_n)_{n\geq 0}$ is an increasing sequence of positive numbers to be
determined later and $M,R > 0$ the constants of \cref{pro:7,pro:1}.

\begin{proposition}
  \label{pro:3}
  Let $\delta_n,\epsilon_n$ be as in \cref{eq:6}. Then there exist test
  functions $(\phi_n)_{n\geq 0}$ for testing
  $H_0 : \psi = \psi_0 \in \mathcal{C}_g(\beta,r,L)$ against
  $H_1 : \|\psi - \psi_1\|_2 \leq \sqrt{2}\delta_n^2$ for any
  $\psi_1 \in \mathcal{F}_n$ such that $\|\psi_1 - \psi_0\|_2 \geq \epsilon_n$,
  with type I and type II errors satisfying
  \begin{equation*}
    P_{\psi_0}^{\eta,n}\phi_n \leq \exp(-6n\delta_n^2),\qquad
    \sup_{\psi\in \mathbb{S}^2\,:\,\|\psi - \psi_1\|_2 \leq \sqrt{2}\delta_n^2}
    P_{\psi}^{\eta,n}(1-\phi_n) \leq \exp(-6n\delta_n^2).
  \end{equation*}
\end{proposition}
\begin{proof}
  By \cref{lem:3}, we deduce that
  $H(P_{\psi_1}^{\eta},P_{\psi_0}^{\eta}) \geq \sqrt{12} \delta_n$. From
  \cref{lem:1}, for any $\psi \in \mathbb{S}^2(\mathbb R)$ with
  $\|\psi - \psi_1\|_2 \leq \sqrt{2}\delta_n^2$ ($\psi$ not necessarily in
  $\mathcal{F}_n$), we have the estimate
  $H(P_{\psi}^{\eta},P_{\psi_1}^{\eta}) \leq \delta_n \leq
  H(P_{\psi_1}^{\eta},P_{\psi_0}^{\eta}) / 2$. Then the conclusion follows from
  \citet[section~7]{GhosalGhoshVanDerVaart2000}.
\end{proof}

The small balls estimate of \cref{pro:3} allows to build the desired test
functions, using the classical approach of the covering of $\mathcal{F}_n$ with
balls of radius $\sqrt{2}\delta_n^2$ in the $L^2(\mathbb R)$ norm
\citep{GhosalGhoshVanDerVaart2000}.

\begin{theorem}
  \label{thm:2}
  Assume that $\psi_0 \in \mathcal{C}_g(\beta,r,L)$ for $\beta,L > 0$ and
  $r\in (0,1)$, and let $\epsilon_n,\delta_n$ be as in \cref{eq:6}. Let
  $N(\sqrt{2}\delta_n^2,\mathcal{F}_n,\|\cdot\|_2)$ be the number of
  $L^2(\mathbb R)$ balls of radius $\sqrt{2}\delta_n^2$ needed to cover
  $\mathcal{F}_n$. Then there exist test functions $(\phi_n)_{n\geq 0}$ such
  that
  \begin{gather*}
    P_{\psi_0}^{\eta,n}\phi_n \leq
    N(\sqrt{2}\delta_n^2,\mathcal{F}_n,\|\cdot\|_2) \exp(-6n\delta_n^2),\
    \mathrm{and}\\
    \sup_{\psi\in \mathcal{F}_n\,:\,\|\psi - \psi_0\|_2 \geq \epsilon_n}
    P_{\psi}^{\eta,n}(1-\phi_n) \leq \exp(-6n\delta_n^2).
  \end{gather*}
\end{theorem}

\subsection{Conclusion of the proof}
\label{sec:conclusion-proof}

Let summarize what we've done so far, and finalize the proof of \cref{thm:5}. In
\cref{lem:8} in appendix, we state sufficient conditions to finish the proof of
our main theorem; these conditions involve two parts. First, proving that for a
suitable sequence $\delta_n \rightarrow 0$ with $n\delta_n^2 \rightarrow$ our
prior puts enough probability mass on the balls $B_n(\delta_n)$ and; the
construction of tests functions with sufficiently rapidly decreasing type I and
type II errors for testing $H_0 : \psi = \psi_0$ against
$H_1 : \|\psi - \psi_0\|_2 \geq \epsilon_n$, for those $\psi$ in a set
$\mathcal{F}_n$ of prior probability $1 - \exp(-6n\delta_n^2)$.

For the prior considered here, we found in \cref{thm:3} that $\delta_n$ must
satisfy $n\delta_n^2 \geq C(\log n)^{b_1/r}$ for some $C > 0$, otherwise the
so-called Kullback-Leilbler condition is not met. Regarding the construction of
tests, this involved to build explicitly the sets $\mathcal{F}_n$ in
\cref{sec:constr-tests}. From that construction and \cref{eq:6}, we deduce that
the required test functions exist, if for some constants $K_1,K_2 > 0$ and a
sequence $u_n \rightarrow \infty$
\begin{equation}
  \label{eq:17}
  \delta_n^2 \leq  \frac{K_1\exp(-2\gamma u_n^2)\epsilon_n^2}{(\log
    n)^{2/r}},\quad \epsilon_n^2 \geq
    K_2 (\log n)^{2a_4} \exp(-\beta u_n^r).
\end{equation}
Since we must also have $n\delta_n^2 \geq C(\log n)^{b_1/r}$, we
deduce that the sequence $(u_n)_{n\geq 1}$ should satisfy, for a suitable
constant $C' > 0$,
\begin{equation*}
    \beta u_n^r + 2 \gamma u_n^2 - 2a_5(\log n)^{s/2}
    \leq \log C' + \log n - r^{-1}(2 + b_1 - 2r a_4) \log \log n.
\end{equation*}
Finally, we can take,
\begin{equation*}
  u_n^2 = \frac{\log n}{2\gamma}  - O( (\log n)^{r/2} )
\end{equation*}
and the conclusion of the proof follows by \cref{eq:17}.

\appendix

\section{Proof of \texorpdfstring{\cref{thm:4}}{theorem \ref{thm:4}}}
\label{sec:proof-thmindepclass}

We need some subsidiaries results to prove the \cref{thm:4}.
\begin{proposition}
  \label{pro:11}
  For all $\beta > 0$, all $0 \leq r \leq 1$ and all $x,y \in \mathbb R^2$, it
  holds $\exp(\beta \|x + y\|^r) \leq \exp(\beta\|x\|^r)\exp(\beta \|y\|^r)$.
\end{proposition}
\begin{proof}
  This follows from the trivial estimate
  \begin{align*}
    \|x + y\|^r
    &\leq (\|x\| + \|y\|)^r = \|x\|(\|x\|+\|y\|)^{r-1} +
      \|y\|(\|x\|+\|y\|)^{r-1}\\
    &\leq \|x\|\|x\|^{r-1} + \|y\|\|y\|^{r-1} = \|x\|^r + \|y\|^r. \qedhere
  \end{align*}
\end{proof}

The next lemma is about the change of window in the STFT; its proof is given for
arbitrary $g\in \mathcal{S}(\mathbb R)$ and $\psi \in \mathcal{S}'(\mathbb R)$
in \citet[lemma~11.3.3]{Groechenig}. The proof is identical when
$g,\psi\in L^2(\mathbb R)$, since it essentially rely on a duality
argument. Note, however, that the class of windows and functions that we are
considering are subset of $\mathcal{S}(\mathbb R)$.
\begin{lemma}
  \label{lem:10}
  Let $g_0,g,h \in L^2(\mathbb R)$ such that
  $\langle h,\, g \rangle \ne 0$ and let $\psi \in L^2(\mathbb R)$. Then
  $|V_{g_0}\psi(x,\omega)| \leq |\langle h,\, g\rangle|^{-1}(|V_g\psi| *
  |V_{g_0}h|)(x,\omega)$ for all $(x,\omega) \in \mathbb R^d$.
\end{lemma}
\begin{proof}
  From \citet[corollary~3.2.3]{Groechenig}, for those
  $g,h \in L^2(\mathbb R)$ with $\langle h,\, g \rangle \ne 0$, we
  have the \textit{inversion formula}
  $\psi = \langle h,\, g\rangle^{-1}\int V_g\psi(x,\omega)\,
  M_{\omega}T_xh \, d\omega dx$ for all $\psi \in L^2$. Applying $V_{g_0}$
  both sides
  \begin{equation*}
    V_{g_0}\psi(x',\omega')
    = \frac{1}{\langle h,\, g\rangle}\int_{\mathbb R^2}V_g\psi(x,\omega)
    V_{g_0}(M_{\omega}T_xh)(x',\omega')\,d\omega dx.
  \end{equation*}
  The conclusion follows because
  $|V_{g_0}(M_{\omega}T_xh)(x',\omega')| = |V_{g_0}h(x' - x, \omega' -
  \omega)|$.
\end{proof}

Finally, we have the sufficient material to establish the independence of the
class $\mathcal{C}_g(\beta,r,L)$ with respect to the choice of the window
function $g$, as soon as $g$ is suitably well behaved.

\begin{proof}[Proof of \cref{thm:4}]
  Using \cref{lem:10}, we have that
  $|V_{g_0}\psi| \leq \|g\|_2^{-2}|V_g\psi|*|V_{g_0}g|$. Then, because $r < 1$
  by assumption,
  \begin{multline*}
    \int_{\mathbb R^2}|V_{g_0}\psi(z)| \exp(\beta \|z\|^r)\,dz\\
    \begin{aligned}
      &\leq \int_{\mathbb R^2}(|V_g\psi(z)|*|V_{g_0}g(z)| \exp(\beta
      \|z\|^r)\,dz\\
      &\leq \iint_{\mathbb R^2}\int_{\mathbb R^2} |V_g\psi(u)\exp(\beta \|u\|^r)
      |V_{g_0}g(z -u)|
      \exp(\beta \|z - u\|^r)\,du dz\\
      &\leq \int_{\mathbb R^2}|V_g\psi(u)| \exp(\beta \|u\|^r)\,du \int_{\mathbb
        R^2}|V_{g_0}g(u)|\exp(\beta \|u\|^r)\,du,
    \end{aligned}
  \end{multline*}
  where we've used Young's inequality and the first estimate of
  \cref{pro:11}. We have by \citet[corollary~3.10]{GroechenigZimmermann2004}
  that $V_{g_0}g \in \mathcal{S}_1^1(\mathbb R^2)$, thus the second
  integral in the rhs of the last equation is bounded.
\end{proof}

\section{Proofs of Kullback-Leibler neighborhoods prior mass}
\label{sec:proofs-kullb-leibl}

\subsection{Proof of \texorpdfstring{\cref{thm:1}}{lemma \ref{thm:1}}}
\label{sec:proof-thm1-lem}

To prove \cref{thm:1}, we need the following intermediate lemmas, relating the
smoothness of $\psi$ to the tails of the Wigner density of $\psi$.
\begin{lemma}
  \label{lem:5}
  Let $\psi \in \mathcal{C}_g(\beta,r,L)$ with $\beta,L > 0$ and $r\in
  (0,1)$. Then,
  \begin{equation*}
    \int_{\mathbb R^2} |W_{\psi}(z)|\, \exp(\beta \|2z\|^r)dz \leq L^2.
  \end{equation*}
\end{lemma}
\begin{proof}
  Let $\breve{\psi}(x) = \psi(-x)$. Then from the definition of $V_g\psi$ and
  $W_{\psi}$ we have that
  $W_{\psi}(x,\omega) = 2 e^{4\pi i \omega x}
  V_{\breve{\psi}}\psi(2x,2\omega)$. By \cref{lem:10} (with
  $|\langle g,\, g\rangle | = \|g\|_2^2 = 1$), \cref{pro:11}, and Young's
  inequality,
  \begin{align*}
    \int |W_{\psi}(z/2)|\, \exp(\beta \|z\|^r)dz
    &\leq 2\int (|V_g\psi| * |V_{\breve{\psi}}g|)(z)\, \exp(\beta \|z\|^r)dz\\
    &\leq 2\iint |V_g\psi(u)|\exp(\beta \|u\|^r)
      |V_{\breve{\psi}}g(z-u)|\exp(\beta \|z - u\|^r)\,dudz\\
    &\leq 2\int |V_g\psi(z)|\, \exp(\beta \|z\|^r)dz \times \int
      |V_{\breve{\psi}}g(z)|\, \exp(\beta \|z\|^r)dz.
  \end{align*}
  Moreover, a straightforward computation shows that
  \begin{equation*}
    V_{\breve{\psi}}g(x,\omega) = e^{-2\pi i \omega x}
    \overline{V_g\psi(x,-\omega)},
  \end{equation*}
  which concludes the proof.
\end{proof}

\begin{lemma}
  \label{lem:6}
  Let $\psi \in \mathcal{C}_g(\beta,r,L)$, with $\beta,L > 0$ and $r\in
  (0,1)$. Then,
  \begin{equation*}
    \sup_{\theta}\int_{\mathbb R} p_{\psi}(x,\theta)\, \exp(2\beta |x|^r)dx
    \leq L^2.
  \end{equation*}
\end{lemma}
\begin{proof}
  From the definition of $p_{\psi}$,
  \begin{equation*}
    \int_{\mathbb R} p_{\psi}(x,\theta)\,e^{2\beta |x|^r}dx
    = \int_{\mathbb R^2}
    W_{\psi}(x\cos\theta - \xi \sin\theta,x\sin\theta + \xi \cos\theta)\,
    e^{2\beta |x|^r}d\xi dx.
  \end{equation*}
  Performing the change of variable $(x,\xi) \mapsto (x \cos\theta + \xi
  \sin\theta, -x\sin\theta + \xi \cos\theta)$, we arrive at
  \begin{equation*}
    \int_{\mathbb R} p_{\psi}(x,\theta)\,e^{2\beta |x|^r}dx
    = \int_{\mathbb R^2} W_{\psi}(x,\xi)\,
    e^{2\beta |x\cos\theta +\xi\sin \theta|^r}d\xi dx.
  \end{equation*}
  But for all $r\in (0,1)$, by the triangle inequality and Hölder's inequality
  \begin{equation*}
    | x\cos\theta + \xi \sin\theta|^r
    \leq (|x \cos \theta| + |\xi \sin\theta|)^r \leq (|x| + |\xi|)^r
    \leq 2^{r/2}(x^2 + \xi^2)^{r/2}.
  \end{equation*}
  Then
  \begin{equation*}
    \int_{\mathbb R} p_{\psi}(x,\theta)\,e^{2\beta |x|^r}dx
    \leq \int_{\mathbb R^2} |W_{\psi}(z)|
    \exp\left(\beta\|2z\|^r\right) \, dz,
  \end{equation*}
  and the conclusion follows from \cref{lem:5}.
\end{proof}

\begin{lemma}
  \label{lem:7}
  For all $\beta,L > 0$ and $r \in (0,1)$ there is a constant
  $C(\beta,r,\eta) > 0$ such that if $\psi \in \mathcal{C}_g(\beta,r,L)$ we have
  $\sup_{\theta}\int_{\mathbb R} p_{\psi}^{\eta}(y,\theta)\, \exp(2\beta
  |y|^r)dy \leq C(\beta,r,\eta)L^2$.
\end{lemma}
\begin{proof}
  Using Fubini's theorem twice and the estimate $|u+x|^r \leq |u|^r + |x|^r$,
  \begin{align*}
    \int p_{\psi}^{\eta}(y,\theta)\, e^{2\beta |y|^r}dy
    &= \sqrt{\frac{\pi}{\gamma}} \iint
      p_{\psi}(x,\theta)\exp\left\{-\frac{\pi^2(x - y)^2}{\gamma}
      \right\}\,dx\, e^{2\beta|y|^r}dy\\
    &=\sqrt{\frac{\pi}{\gamma}}
      \iint p_{\psi}(x,\theta)\exp\left\{-\frac{\pi^2 u^2}{\gamma} \right\}
      \exp\left(2\beta |u + x|^r\right)\,du dx\\
    &\leq
      \sqrt{\frac{\pi}{\gamma}}
      \int p_{\psi}(x,\theta)\,e^{2\beta|x|^r}dx
      \int \exp\left\{-\frac{\pi^2 u^2}{\gamma} + 2\beta |u|^r
      \right\}\,du.
  \end{align*}
  The conclusion follows from \cref{lem:6}.
\end{proof}

From the lemmas above the proof of \cref{thm:1} is relatively straightforward,
we give it here for the sake of completeness.
\begin{proof}[Proof of \cref{thm:1}]
  We begin with the obvious estimate
  $P_{\psi}^{\eta,n}(\Omega_n^c) \leq n P_{\psi}^{\eta}(E_n^c)$. The proof is
  finished by noticing that
  \begin{align*}
    P_{\psi}^{\eta}(E_n^c)
    &= \int_{E_n^c}p_{\psi}^{\eta}(y,\theta)\,
    e^{2\beta |y|^r}e^{-2\beta |y|^r}dyd\theta\\
    &\leq n^{-2} \int
    p_{\psi}^{\eta}(y,\theta)\,e^{2\beta|y|^r}dyd\theta\\
    &\leq 2\pi C(\beta,r,\eta)L^2 n^{-2},
  \end{align*}
  because of \cref{lem:7}.
\end{proof}

\subsection{Proofs regarding approximation theory}
\label{sec:proofs-regard-appr}

\begin{proof}[Proof of \cref{lem:2}]
  For all $\psi \in \mathcal{M}_n(Z,U)$ we have the following estimate. Because
  $(\varphi_{lm})$ is an orthonormal base of $L^2(\mathbb R)$,
  \begin{align*}
    \|\psi - \psi_Z\|_2^2
    &= \sum_{(l,m)\in \Lambda_Z}| p_{lm}e^{i\zeta_{lm}} - c_{lm}|^2\\
    &\leq 2\sum_{(l,m)\in \Lambda_Z}|p_{lm} - |c_{lm}||^2
      + 2\sum_{(l,m)\in \Lambda_Z}|\zeta_{lm} - \arg c_{lm}|^2
      \leq 4U^2.
  \end{align*}
  Then the conclusion follows using $\|\psi - \psi_0\|_2 \leq \|\psi -
  \psi_Z\|_2 + \|\psi_Z - \psi_0\|_2$ and \cref{eq:20}.
\end{proof}

\begin{proof}[Proof of \cref{lem:4}]
  Recall that
  $p_{\psi}^{\eta}(y,\theta) = [p_{\psi}(\cdot,\theta)*G_{\gamma}](y)$. We have
  the obvious bound
  \begin{equation*}
    p_{\psi}^{\eta}(y,\theta)
    = \int_{-\infty}^{+\infty}p_{\psi}(x,\theta)G_{\gamma}(y -
    x)\,dx
    \geq
    \int_{-D_n^{\beta,r}}^{+D_n^{\beta,r}}
    p_{\psi}(x,\theta) G_{\gamma}(y - x)\,dx.
  \end{equation*}
  Then for all $(y,\theta) \in E_n$ (\textit{i.e.}  $|y| \leq D_n^{\beta,r}$) it
  follows from the definition of $G_{\gamma}$ that
  $p_{\psi}^{\eta}(y,\theta) \geq G_{\gamma}(2D_n^{\beta,r})P_{\psi}(|X| \leq
  D_n^{\beta,r} \mid \theta) / (2\pi)$. From \cref{pro:2} in appendix, the
  latter implies for $n$ large enough that for all $\psi \in \mathcal{M}_n$ it
  holds $p_{\psi}^{\eta}(y,\theta) \geq G_{\gamma}(2D_n^{\beta,r})/(4\pi)$
  whenever $(y,\theta)\in E_n$. Since $\psi_0 \in \mathcal{C}_g(\beta,r,L)$,
  which is a subset of the Schwartz space $\mathcal{S}(\mathbb R)$, and since
  the Radon transform maps $\mathcal{S}(\mathbb R)$ onto a subset of
  $\mathcal{S}(\mathbb R \times [0,2\pi])$ by \citet[theorem~2.4]{Helgason2011},
  we deduce that there is a constant $C = C(\psi_0,\eta) > 0$ such that for all
  $\psi \in \mathcal{M}_n(Z,U)$,
  \begin{equation*}
    \frac{p_{\psi_0}^{\eta}(y,\theta)}{p_{\psi}^{\eta}(y,\theta)}
    \leq C \exp\left\{\frac{4\pi^2}{\gamma}\left( \frac{\log n}{\beta}
      \right)^{2/r} \right\} =: \lambda_n^{-1},\quad \forall (y,\theta) \in
    E_n.
  \end{equation*}
  The proof now follows similar lines as
  \citet[lemma~B2]{ShenTokdarGhosal2013}. The function
  $r : (0,\infty) \rightarrow \mathbb R$ defined implicitly by
  $\log x = 2(x^{1/2}-1) - r(x)(x^{1/2}-1)^2$ is nonnegative and
  decreasing. Thus we obtain,
  \begin{multline}
    \label{eq:21}
    \int_{E_n}p_{\psi_0}^{\eta}\log
    \frac{p_{\psi_0}^{\eta}}{p_{\psi}^{\eta}}\\
    \begin{aligned}
      &= -2\int_{E_n}p_{\psi_0}^{\eta}\left(
        \sqrt{\frac{p_{\psi}^{\eta}}{p_{\psi_0}^{\eta}}} - 1\right) +
      \int_{E_n}p_{\psi_0}^{\eta}r \left(
        \frac{p_{\psi}^{\eta}}{p_{\psi_0}^{\eta}}\right)\left(
        \sqrt{\frac{p_{\psi}^{\eta}}{p_{\psi_0}^{\eta}}} - 1\right)^{2}\\
      &\leq
      \begin{multlined}[t][0.8\textwidth]
        2\left(1 - \int \sqrt{p_{\psi_0}^{\eta} p_{\psi}^{\eta}} \right) - 2
        P_{\psi_0}^{\eta}(E_n^c)\\
        + 2 \int_{E_n^c}\sqrt{p_{\psi_0}^{\eta} p_{\psi}^{\eta}} + r(\lambda_n)
        \int_{E_n}\left( \sqrt{p_{\psi}^{\eta}} - \sqrt{p_{\psi_0}^{\eta}}
        \right)^2
      \end{multlined}\\
      &\leq 2 H^2(P_{\psi}^{\eta},P_{\psi_0}^{\eta})\left(1 + r(\lambda_n)
      \right) + 2P_{\psi_0}^{\eta}(E_n^c)^{1/2} P_{\psi}^{\eta}(E_n^c)^{1/2},
    \end{aligned}
  \end{multline}
  where the last line follows from Hölder's inequality. Also, proceeding as in
  the proof of \citet[lemma~B2]{ShenTokdarGhosal2013} we find that
  \begin{equation}
    \label{eq:22}
    \int_{E_n}p_{\psi_0}^{\eta}\left(\log
      \frac{p_{\psi_0}^{\eta}}{p_{\psi}^{\eta}}\right)^2
    \leq H^2(P_{\psi}^{\eta},P_{\psi_0}^{\eta})\left( 12 + 2 r(\lambda_n)^2
    \right).
  \end{equation}
  Note that $r(x) \leq \log x^{-1}$ for $x$ small enough, and by \cref{thm:1},
  \begin{equation}
    \label{eq:23}
    P_{\psi_0}^{\eta}(E_n^c)^{1/2} P_{\psi}^{\eta}(E_n^c)^{1/2} \leq
    P_{\psi_0}^{\eta}(E_n^c)^{1/2} \leq \sqrt{2\pi C(\beta,r,\eta)}Ln^{-1}.
  \end{equation}
  Then we deduce from \cref{eq:21,eq:22,eq:23,lem:1} that for $n$ large enough,
  provided $\delta_n^2\geq 4\sqrt{2\pi C(\beta,r,\eta)}L n^{-1}$,
  \begin{equation*}
    B_n(\delta_n) \supset
    \Set*{P_{\psi}^{\eta} \given \psi \in \mathbb S^2,\quad
      \|\psi - \psi_0\|_2 \leq \frac{\gamma^2}{48\sqrt{2}\pi^4}
      \left(\frac{\beta}{\log n}\right)^{4/r}\delta_n^2 }.
  \end{equation*}
  Then the conclusion follows from \cref{lem:2}.
\end{proof}

\subsection{Proof of the lower bound}
\label{sec:proof-lower-bound}

\begin{proof}[Proof of \cref{thm:3}]
  Let $C_1,C_2 > 0$ be the constants of \cref{lem:4}, let
  $U_n = C_1(\log n)^{-4/r}\delta_n^2$ and $Z_n$ be the smaller integer larger
  than $C_2(\log \delta_n^{-1})^{1/r}$. Then by \cref{lem:4}
  $\Pi(B_n(\delta_n)) \geq \Pi(\mathcal{M}_n(Z_n,U_n))$, and
  \begin{multline*}
    \Pi(\mathcal{M}_n(Z_n,U_n)) \geq  P_Z(Z = Z_n)
    G\left(\textstyle\sum_{(l,m)\in \Lambda_Z}|p_{lm} - |c_{lm}| |^2 \leq
      U_n^2 \mid Z\right)\\ \times P_{\zeta}\left(\textstyle\sum_{(l,m)\in
        \Lambda_Z}|\zeta_{lm} - \arg c_{lm}|^2 \leq U^2 \mid Z \right).
  \end{multline*}
  Note that by \cref{lem:9} the sequence $(|c_{lm}|)_{(l,m)\in \lambda_Z}$ is in
  $\Delta_Z^w(\beta,r,C(\beta,r)L)$. Hence, using the assumptions of
  \cref{sec:assumptions}, we have for $n$ large enough
  \begin{equation*}
    \Pi(\mathcal{M}_n(Z_n,U_n))
    \gtrsim
    \exp\left\{ -a_1 Z_n^{b_1} - (a_0 + a_3)Z_n^{b_1 - r} \log U_n^{-2}\right\}.
  \end{equation*}
  We deduce from the above the existence of a constant $K > 0$ not depending on
  $n$, such that for $n$ large enough,
  \begin{align*}
    \Pi(B_n(\delta_n))
    &\gtrsim
      \exp\left\{ -K (\log \delta_n^{-1})^{b_1/r}
      -K (\log \delta_n^{-1})^{b_1/r - 1}\left( \log \delta_n^{-1} + \log\log n
      \right) \right\}\\
    &\gtrsim \exp( - n \delta_n^2).
  \end{align*}
  Then the conclusion of the theorem follows since we assume
  $n\delta_n^2 = C(\log n)^{b_1/r}$ for a suitable constant $C > 0$.
\end{proof}

\section{Proofs of tests construction}
\label{sec:proofs-tests-constr}

\subsection{Proofs regarding the sieve}
\label{sec:proofs-sieve-constr}

\begin{proof}[Proof of \cref{pro:5}]
  Let $Z_n$ be the smaller integer larger than $h (\log n)^{1/r}$. Clearly
  $\psi \sim \Pi$ is almost-surely in $\mathbb S^2(\mathbb R)$. Then if
  $c > 0$ is large enough we have the bound
  \begin{align*}
    \Pi(\mathcal{F}_n^c)
    &\leq P_Z\left(Z > h(\log n)^{1/r}\right)
      +
      G\left( \mathbf p \notin \Delta_Z^w(\beta,r,c(\log n)^{a_4}) \mid Z \leq
      Z_n  \right)\\
    &\lesssim
      \exp\left(- a_2 h^{b_1}(\log n)^{b_1/r} \right) + \exp\left(-a_5(\log
      n)^{b_5}\right)
  \end{align*}
  which is trivially smaller than a multiple constant of $\exp(-6n\delta_n^2)$
  when $h$ is as large as in the proposition, and because $b_5 > b_1/r$ by
  assumption.
\end{proof}

\begin{proof}[Proof of \cref{pro:6}]
  We use the classical argument that
  $N(\sqrt{2}\delta_n^2, \mathcal{F}_n, \|\cdot\|_2)$ is bounded by the
  cardinality of a $\sqrt{2}\delta_n^2$-net over $\mathcal{F}_n$ is the
  $\|\cdot\|_2$ distance \citep{ShenTokdarGhosal2013}. We compute the
  cardinality of such $\sqrt{2}\delta_n^2$-net as follows. Let
  $Z_n := h(\log n)^{1/r}$, $\widehat{\mathcal{P}}$ be a $\delta_n^2$-net over
  the simplex $\Delta_{Z_n}$ in the $\ell_2$ distance, and let
  $\widehat{\mathcal{O}}$ be a $\delta_n^2$-net over $[0,2\pi]$ in the euclidean
  distance. Then define
  \begin{equation*}
    \mathcal{N}_n := \Set*{
      \psi \in \mathbb S^2(\mathbb R)
      \given
      \begin{array}{l}
        \widetilde{\psi} = \sum_{(l,m)\in \Lambda_{Z_n}}
        v_{lm} e^{i\zeta_{lm}}\, \varphi_{lm},\\
        (v_{lm})_{(l,m)\in \Lambda_{Z_n}} \in \widehat{\mathcal{P}},\quad
        \forall (l,m)\in \Lambda_{Z_n}\ :\ \zeta_{lm} \in \widehat{\mathcal{O}}
      \end{array}
    }.
  \end{equation*}
  For all $\psi \in \mathcal{F}_n$ we have
  $\psi = \sum_{(l,m)\in \Lambda_{Z_n}} q_{lm} e^{i\zeta_{lm}}\, \varphi_{lm}$,
  with $q_{lm} = p_{lm}$ for those $(l,m) \in \Lambda_{Z}$, $Z\leq Z_n$, and
  $q_{lm} = 0$ otherwise. Since $(\varphi_{lm})$ is an orthonormal base of
  $L^2(\mathbb R)$, we have $\sum_{(l,m)\in \Lambda_{Z_n}}q_{lm}^2 =1$, and we
  can find a function $\psi' \in \mathcal{N}_n$ with
  $\psi' = \sum_{(l,m)\in \Lambda_{Z_n}} q_{lm}'e^{i\zeta_{lm}'}\, \varphi_{lm}$
  such that $\sum_{(l,m)\in \Lambda_{Z_n}}|q_{lm}' - q_{lm}|^2 \leq \delta_n^4$,
  and $|\zeta_{lm}' - \zeta_{lm}| \leq \delta_n^2$ for all
  $(l,m)\in \Lambda_{Z_n}$. Using standard arguments, we have
  \begin{align*}
    \|\psi' - \psi\|_2^2
    &= \sum_{(l,m)\in \Lambda_Z}\left|
      q_{lm}'e^{i\zeta_{lm}'} - q_{lm}e^{i\zeta_{lm}} \right|^2\\
    &\leq 2 \sum_{(l,m)\in \Lambda_Z}\left|
      q_{lm}' - q_{lm} \right|^2
    + 2 \sum_{(l,m)\in \Lambda_Z}q_{lm}^2\left|
      e^{i\zeta_{lm}'} - e^{i\zeta_{lm}} \right|^2 \leq 4\delta_n^4.
  \end{align*}
  Thus $\mathcal{N}_n$ is a $2\delta_n^2$ over $\mathcal{F}_n$ in
  the $\|\cdot\|_2$ norm. Moreover, the cardinality of $\mathcal{N}_n$ is upper
  bounded by
  $|\widehat{\mathcal{P}}|\times |\widehat{\mathcal{O}}|^{|\Lambda_{Z_n}|}$,
  which is in turn bounded by
  \begin{equation*}
    C \left( \frac{1}{\delta_n^4} \right)^{|\Lambda_{Z_n}|}
    \left( \frac{2\pi}{\delta_n^2}\right)^{|\Lambda_{Z_n}|},
  \end{equation*}
  for a constant $C > 0$. Clearly, the cardinality of a $\sqrt{2}\delta_n^2$-net
  over $\mathcal{F}_n$ in the $\|\cdot\|_2$ distance satisfy the same bound,
  eventually for a different constant $C$. Therefore, for a suitable constant
  $K > 0$, when $n$ is large enough.
  \begin{equation*}
    N(\sqrt{2}\delta_n^2, \mathcal{F}_n, \|\cdot\|_2)
    \lesssim \exp\left\{ K |\Lambda_{Z_n}| \log \frac{1}{\delta_n}
      \right\}
    \lesssim \exp\left\{ K h^2(\log n)^{1 + 2/r} \right\}.
  \end{equation*}
  The conclusion follows because $b_1 > 2 + r$.
\end{proof}

\begin{proof}[Proof of \cref{pro:7}]
  The bound is obvious for those $\psi \in \mathcal{F}_n$ with $Z=0$. For
  $Z \geq 1$, we have from the definition of the Wigner transform
  (\cref{eq:11}), for an arbitrary $\psi \in \mathcal{F}_n$,
  \begin{equation*}
    W_{\psi}(x,\omega) = \sum_{(l,m)\in \Lambda_Z}\sum_{(j,k)\in \Lambda_Z}
    p_{lm} p_{jk}
    e^{i(\zeta_{lm} - \zeta_{jk})}
    \int_{\mathbb R} \varphi_{lm}(x + t/2) \overline{\varphi_{jk}(x -
      t/2)}e^{-2\pi i \omega t}\,dt.
  \end{equation*}
  Using the expression of $\varphi_{lm}$ from \cref{eq:32}, it follows
  \begin{multline*}
    \varphi_{lm}(x + t/2) \overline{\varphi_{jk}(x - t/2)}
    = c_lc_j T_mM_l\varphi(x +t/2) \overline{T_kM_j\varphi(x - t/2)}\\
    \begin{aligned}
      &+ (-1)^{2k+j}c_lc_j T_mM_l \varphi(x + t/2) \overline{T_kM_{-j} \varphi(x
        - t/2)}\\
      &+ (-1)^{2m+l}c_lc_j T_mM_{-l}\varphi(x + t/2)
      \overline{T_kM_j\varphi(x-t/2)}\\
      &+ (-1)^{2m+l}(-1)^{2k+j}c_lc_j T_mM_{-l}\varphi(x+t/2)
      \overline{T_kM_{-j}\varphi(x - t/2)}.
    \end{aligned}
  \end{multline*}
  Recalling that $T_x\varphi(y) = \varphi(y - x)$ and
  $M_{\omega}\varphi(y) = e^{2\pi i \omega y}\varphi(y)$, it follows
  \begin{multline*}
    \int_{\mathbb R} T_mM_l\varphi(x + t/2) \overline{T_kM_j\varphi(x -
      t/2)}e^{-2\pi i \omega t}\,dt\\
    \begin{aligned}
      &= \int_{\mathbb R}
      e^{2\pi i l(x+t/2 - m)}\varphi(x+t/2-m)
      e^{-2\pi i j(x - t/2 - k)} \overline{\varphi(x-t/2-k)}e^{-2\pi i \omega
        t}\,dt\\
      &= 2 e^{4\pi i \omega (x - m) - 2\pi i j(2x - m - k)} \int_{\mathbb R}
      \varphi(u)\overline{\varphi(-u + 2x - m - k)}
      e^{-2\pi i u (2\omega - l - j)}\,du\\
      &= 2 e^{4\pi i \omega (x - m) - 2\pi i j(2x - m - k)}
      V_{\breve{\varphi}}\varphi(2x - m - k, 2\omega - l - j).
    \end{aligned}
  \end{multline*}
  Thus, we deduce the following expression for the Wigner transform of an
  arbitrary function $\psi \in \mathcal{F}_n$.
  \begin{multline*}
    W_{\psi}(x,\omega) = \sum_{(l,m)\in \Lambda_Z}\sum_{(j,k)\in \Lambda_Z}
    p_{lm}p_{jk}e^{i(\zeta_{lm} - \zeta_{jk})} \times 2c_lc_j e^{4\pi i \omega
      (x - m) } \Big[\\
    \begin{aligned}
      &\quad e^{- 2\pi i j(2x - m - k)}
      V_{\breve{\varphi}}\varphi(2x - m - k, 2\omega - l - j)\\
      &+ (-1)^{2k+j} e^{2\pi i j(2x - m - k)}
      V_{\breve{\varphi}}\varphi(2x - m - k, 2\omega - l + j)\\
      &+ (-1)^{2m+l} e^{- 2\pi i j(2x - m - k)}
      V_{\breve{\varphi}}\varphi(2x - m - k, 2\omega + l - j)\\
      &+ (-1)^{2m+l}(-1)^{2k+j} e^{2\pi i j(2x - m - k)}
      V_{\breve{\varphi}}\varphi(2x - m - k, 2\omega + l + j) \Big].
    \end{aligned}
  \end{multline*}
  To ease notations, let
  \begin{equation*}
    f(x,\omega;l,m,j,k) := e^{4\pi i \omega(x -m) -2\pi ij(2x - m -k)}
    V_{\breve{\varphi}}\varphi(2x - m -k,2\omega -l - j).
  \end{equation*}
  Letting $\mathscr{R}f(z,\theta)$ denote the Radon transform of $f$, it is easy
  to check that
  $\mathscr F [\mathscr{R}f(\cdot,\theta)](u) = \widehat{f}(u\cos\theta, u\sin
  \theta)$, where $\widehat{f}$ is the Fourier transform with respect to both
  variables of $f$, and $\mathscr F$ the $L^1$ Fourier operator. Note that,
  \begin{multline*}
    \int V_{\breve{\varphi}}\varphi(x,y) e^{\pi i x y} e^{-2\pi i(x\xi_1 + y
      \xi_2)}\,dxdy\\
    \begin{aligned}
      &= \int_{\mathbb R^2}\int_{\mathbb R}\varphi(u)\overline{\varphi(x - u)}
      e^{-2\pi i u y}\,du\, e^{\pi ixy} e^{-2\pi i (x\xi_1 + y\xi_2)} dx dy\\
      &= \iint \varphi(u) e^{-\pi i u y}\int \overline{\varphi(t)}e^{2\pi i
        t(\xi_1 - y/2)}dt\,
      e^{-2\pi i u\xi_1 -2\pi i y\xi_2}dy\,du\\
      &= 2e^{4\pi i \xi_1 \xi_2} \int \widehat{\varphi}(t)
      \overline{\widehat{\varphi}(t-2\xi_{1})}
      e^{-4\pi i t \xi_2}\,dt\\
      &=2e^{4\pi i \xi_1 \xi_2}
      V_{\widehat{\varphi}}\widehat{\varphi}(2\xi_1,2\xi_2).
    \end{aligned}
  \end{multline*}
  Hence,
  \begin{equation*}
    |\widehat{f}(u\cos\theta,u\sin \theta;l,m,j,k)|
    = \frac 12 |V_{\widehat{\varphi}}\widehat{\varphi}(u\cos\theta + j - l,
    u\sin\theta + m - k)|
  \end{equation*}
  By Fourier duality, this implies that
  \begin{equation*}
    \sup_x| \mathscr Rf(\cdot;l,m,j,k)(x,\theta)|
    \leq \frac 12 \int |V_{\widehat{\varphi}}\widehat{\varphi}(u\cos\theta + j
    -l, u\sin \theta + m -k)|\,du
  \end{equation*}
  The function $\varphi$ is in $\mathcal{S}_1^1(\mathbb R)$ by construction.
  From \citet[corollary~3.10]{GroechenigZimmermann2004} we can then find a
  constant $a > 0$ such that
  $|V_{\widehat{\varphi}}\widehat{\varphi}(x,\omega)| \lesssim \exp( -a
  \sqrt{x^2 + \omega^2})$. Moreover,
  \begin{multline*}
    (u\cos\theta + j -l)^2 + (u\sin\theta + m -k)^2\\
    \begin{aligned}
      &= \left( u + (j-l)\cos\theta + (m-k)\sin\theta \right)^2 + \left(
        (m-k)\cos\theta - (j-l)\sin \theta \right)^{2}\\
      &\geq \left( u + (j-l)\cos\theta + (m-k)\sin\theta \right)^2.
    \end{aligned}
  \end{multline*}
  Therefore,
  \begin{equation*}
    \sup_{x,\theta}| \mathscr Rf(\cdot;l,m,j,k)(x,\theta)|
    \lesssim \int \exp(-a |u|)\,du = 2a^{-1}.
  \end{equation*}
  Since the Radon transform is a linear map, we deduce that
  \begin{equation*}
    |p_{\psi}(x,\theta)|
    \lesssim 8a^{-1} \left(\textstyle \sum_{(l,m)\in \Lambda_Z}
    p_{lm} \right)^2 \leq 8a^{-1}|\lambda_Z| \leq 8a^{-1}h^2 (\log n)^{2/r}.
  \end{equation*}
  Now $p_{\psi}^{\eta}(y,\theta) = [p_{\psi}(\cdot,\theta)* G_{\gamma}](y)$, so
  that conclusion of the proposition follows from Young's inequality.
\end{proof}

\begin{proof}[Proof of \cref{pro:1}]
  Using the expression of $\varphi_{lm}$ of \cref{eq:32}, we have
  \begin{equation*}
    V_g\varphi_{lm} = c_l V_g(T_mM_l\varphi)
    + (-1)^{2m+l}c_l V_g(T_mM_{-l}\varphi).
  \end{equation*}
  Since $|V_g(T_mM_l\varphi)(x,\omega)| = |V_g(x - m, \omega - m)|$, it follows
  \begin{equation*}
    |V_g\varphi_{lm}(x,\omega)| \leq c_l|V_g\varphi(x - m, \omega - l)| +
    c_l|V_g\varphi(x - m, \omega + l)|.
  \end{equation*}
  Now pick an arbitrary $\psi \in \mathcal{F}_n$. We have
  \begin{multline*}
    \int_{\mathbb R^2}|V_g\psi(z)|\exp(\beta \|z\|^r) \,dz\\
    \begin{aligned}
      &\leq \sum_{(l,m)\in \Lambda_Z} p_{lm} \int_{\mathbb R^2}
      |V_g\varphi_{lm}(z)| \exp(\beta \|z\|^r)\,dz\\
      &\leq \sum_{(l,m)\in \Lambda_Z} p_{lm} c_l \int_{\mathbb R^2}
      |V_g\varphi(x - m,\omega - l)|\exp\left(\beta
        (x^2+\omega^{2})^{r/2}\right)\, dxd\omega\\
      &\quad + \sum_{(l,m)\in \Lambda_Z} p_{lm} c_l \int_{\mathbb R^2}
      |V_g\varphi(x - m,\omega + l)|\exp\left(\beta
        (x^2+\omega^{2})^{r/2}\right)\, dxd\omega\\
      &\leq 2 \sum_{(l,m)\in \Lambda_Z} p_{lm} \exp\left(\beta (l^2 +
        m^2)^{r/2}\right)
      \int_{\mathbb R^2} |V_g\varphi(z)| \exp(\beta \|z\|^r)\, dz\\
      &\lesssim \sum_{(l,m)\in \Lambda_Z} p_{lm} \exp\left(\beta (l^2 +
        m^2)^{r/2}\right) \lesssim (\log n)^{a_4},
    \end{aligned}
  \end{multline*}
  where the last line follows from
  \citet[corollary~3.10]{GroechenigZimmermann2004}, since both $g$ and $\varphi$
  are in $\mathcal{S}^1_1(\mathbb R)$ and $r < 1$ by assumption. The previous
  estimate show that $\mathcal{F}_n \subset \mathcal{C}_g(\beta,r, L_n)$ with
  $L_n \lesssim (\log n)^{a_4}$. Hence the conclusion follows from \cref{pro:8}.
\end{proof}

\subsection{Proofs of norm equivalence}
\label{sec:proofs-norm-equiv}

\begin{proof}[Proof of \cref{pro:8}]
  Recall that $\mathscr F$ denote the $L^{1}$ Fourier transfom operator. By
  definition of $W_{\psi}$, it holds
  $W_{\psi}(u_1,u_2) = \mathscr{F}[\psi(u_1+\cdot/2)
  \overline{\psi(u_1-\cdot/2)}](u_2)$. Clearly if
  $\psi \in \mathcal{C}_g(\beta,r,L)$ then $W_{\psi} \in L^1(\mathbb R^2)$ by
  \cref{lem:5}. Moreover, for all $u_1 \in \mathbb R$ the mapping
  $t\mapsto \psi(u_1 + t/2) \overline{\psi(u_1 - t/2)}$ is in $L^1(\mathbb R)$
  because of Cauchy-Schwarz inequality and $\psi \in L^2(\mathbb R)$. Then by
  Fourier inversion, we get
  \begin{equation*}
    \int W_{\psi}(u_1,u_2) e^{-2\pi i u_2(-\xi_2)}\,du_2
    = \psi(u_1 + \xi_2/2) \overline{\psi(u_1 - \xi_2/2)}.
  \end{equation*}
  Taking the Fourier transform with respect to $u_1$ yields
  \begin{align*}
    \iint W_{\psi}(u_1,u_2) e^{-2\pi i (u_1\xi_1 + u_2\xi_2)}\, du_1du_2
      &= \int \psi(u_1 - \xi_2/2) \overline{\psi(u_1 + \xi_2/2)} e^{-2\pi i u_1
        \xi_1} \,du_1\\
      &= e^{-\pi i \xi_1 \xi_2} \int \psi(t) \overline{\psi(t + \xi_2)} e^{-2\pi
        i \xi_1t}\,dt.
  \end{align*}
  Hence we proved that
  $\widehat{W}_{\psi}(\xi_1,\xi_2) = e^{-\pi i \xi_1 \xi_2}
  V_{\psi}\psi(-\xi_2,\xi_1)$, at least when
  $\psi \in \mathcal{C}_g(\beta,r,L)$. By \cref{lem:10},
  $|V_{\psi}\psi(-\xi_2,\xi_1)| \leq (|V_g\psi| * |V_{\psi}g|)(-\xi_2,\xi_1)$
  since $\|g\|_2 = 1$. Note that, by \cref{pro:11} we have
  \begin{equation*}
    \exp(\beta (\xi_1^2 + \xi_2^2)^{r/2}) \leq
    \exp(\beta ( (-\xi_2 - u_1)^2 + (\xi_1 - u_2)^2 )^{r/2})
    \exp(\beta (u_1^2 + u_2^2)^{r/2}).
  \end{equation*}
  Also, by Cauchy-Schwarz inequality
  $|\widehat{W}_{\psi}(\xi_1,\xi_2)| \leq \|\psi\|_2^2 = 1$. Therefore, by
  Young's inequality, and because $|V_{\psi}g| = |V_g\psi|$,
  \begin{multline*}
    \iint|\widehat{W}_{\psi}(\xi_1,\xi_2)|^2 \exp(\beta(\xi_1^2 +
    \xi_2^2)^{r/2})\, d\xi_1d\xi_2\\
    \begin{aligned}
      &\leq \iint|\widehat{W}_{\psi}(\xi_1,\xi_2)| \exp(\beta(\xi_1^2 +
      \xi_2^2)^{r/2})\, d\xi_1d\xi_2\\
      &\leq \left(\iint |V_g\psi(\xi_1,\xi_2)| \exp(\beta(\xi_1^2 +
        \xi_2^2)^{r/2})\, d\xi_1d\xi_2\right)^2,
    \end{aligned}
  \end{multline*}
  which concludes the proof.
\end{proof}

\begin{proof}[Proof of \cref{lem:3}]
  The lower bound follows from \cref{lem:1} in \cref{sec:auxiliary-results}. In
  the sequel we let $M_n := Mh^2 (\log n)^{2/r}$ and
  $R_n := R (\log n)^{2a_4}\exp(-\beta u^r)$. To establish the upper bound, we
  first bound the $L^2$ distance between densities by the Hellinger distance. By
  triangular inequality and Young's inequality,
  \begin{multline*}
    |p_{\psi}^{\eta}(y,\theta) - p_{\psi_0}^{\eta}(y,\theta)|^2 \leq
    2\left|\sqrt{p_{\psi}^{\eta}(y,\theta)} \sqrt{p_{\psi_0}^{\eta}(y,\theta)} -
      \sqrt{p_{\psi}^{\eta}(y,\theta)}
      \sqrt{p_{\psi}^{\eta}(y,\theta)}\right|^2\\
    + 2\left|\sqrt{p_{\psi}^{\eta}(y,\theta)} \sqrt{p_{\psi_0}^{\eta}(y,\theta)}
      - \sqrt{p_{\psi_0}^{\eta}(y,\theta)}
      \sqrt{p_{\psi_0}^{\eta}(y,\theta)}\right|^2.
  \end{multline*}
  Taking the integral both sides, under the assumptions of the lemma it comes
  \begin{equation*}
    \iint |p_{\psi}^{\eta}(y,\theta) - p_{\psi_0}^{\eta}(y,\theta)|^2\,dyd\theta
    \leq 2(C_0 + M_n)H^2(P_{\psi}^{\eta},P_{\psi_0}^{\eta}).
  \end{equation*}
  Recall that $\mathscr{F}$ denote the $L^1$-Fourier transform operator. Then by
  Parseval-Plancherel formula we can rewrite
  \begin{equation*}
    \iint |\mathscr{F}[p_{\psi}^{\eta}(\cdot,\theta)](\xi) -
    \mathscr{F}[p_{\psi_0}^{\eta}(\cdot,\theta)](\xi)|^2\,d\xi d\theta
    \leq 2(C_0 + M_n)H^2(P_{\psi}^{\eta},P_{\psi_0}^{\eta}).
  \end{equation*}
  Recalling that
  $p_{\psi}^{\eta}(y,\theta) = [p_{\psi}(\cdot,\theta)*G_{\gamma}](y)$, where
  $\mathscr{F}[G_{\gamma}] = \widehat{G}_{\gamma}$, we deduce that
  $\mathscr{F}[p_{\psi}^{\eta}(\cdot,\theta)](\xi) =
  \mathscr{F}[p_{\psi}(\cdot,\theta)](\xi)
  \widehat{G}_{\gamma}(\xi)$. Therefore,
  \begin{equation*}
    \iint |\mathscr{F}[p_{\psi}(\cdot,\theta)](\xi) -
    \mathscr{F}[p_{\psi_0}(\cdot,\theta)](\xi)|^2
    |\widehat{G}_{\gamma}(\xi)|^{2} \,d\xi d\theta
    \leq 2(C_0 + M_n)H^2(P_{\psi}^{\eta},P_{\psi_0}^{\eta}).
  \end{equation*}
  Using that
  $\mathscr{F}[p_{\psi}(\cdot,\theta)](\xi) =
  \widehat{W}_{\psi}(\xi\cos\theta,\xi\sin\theta)$, and performing the suitable
  change of variables, we arrive at
  \begin{equation*}
    \int_{\mathbb R^2} |\widehat{W}_{\psi}(z) -
    \widehat{W}_{\psi_0}(z)|^2 |\widehat{G}_{\gamma}(\|z\|)|^{2} \,dz \leq
    2(C_0 + M_n)H^2(P_{\psi}^{\eta},P_{\psi_0}^{\eta}).
  \end{equation*}
  Now, using that the Fourier transform is isometric from $L^2(\mathbb R)$ onto
  itself, and that the Wigner transform is isometric from $L^2(\mathbb R)$ onto
  $L^2(\mathbb R^2)$, by \citet[proposition~4.3.2]{Groechenig}, we write
  \begin{align*}
    \|\psi - \psi_0\|_2^2
    &= \int_{\mathbb R^2}| \widehat{W}_{\psi}(z) -
      \widehat{W}_{\psi_0}(z)|^2\,dz\\
    &=\int_{\Set{\|z\| \leq u}}| \widehat{W}_{\psi}(z) -
      \widehat{W}_{\psi_0}(z)|^2\,dz + \int_{\Set{\|z\| > u}}|
      \widehat{W}_{\psi}(z) - \widehat{W}_{\psi_0}(z)|^2\,dz\\
    &\leq\begin{multlined}[t][.8\textwidth]
      \frac{1}{|\widehat{G}_{\gamma}(u)|^2}
      \int_{\mathbb R^2}| \widehat{W}_{\psi}(z) -
      \widehat{W}_{\psi_0}(z)|^2 |\widehat{G}_{\gamma}(\|z\|)|^2\,dz\\
      + \int_{\Set{\|z\| > u}}
      |\widehat{W}_{\psi}(z) - \widehat{W}_{\psi_0}(z)|^2\,dz.
      \end{multlined}
  \end{align*}
  Under the hypothesis of the lemma, the second term in the rhs of the last
  equation is bounded by $4R_n$ when $n$ is large, because by \cref{pro:8} we
  have
  \begin{align*}
    \int_{\Set{\|z\|> u}}|\widehat{W}_{\psi_0}(z)|^2\,dz
    &= \int_{\Set{\|z\|> u}}|\widehat{W}_{\psi_0}(z)|^2\, e^{\beta \|z\|^r}
    e^{-\nu \|z\|^r} dz\\
    &\leq e^{-\beta u^r} \int_{\mathbb R^2}|\widehat{W}_{\psi_0}(z)|^2\,
    e^{\beta \|z\|^r}dz \leq L^2 e^{-\beta u^r}.
  \end{align*}
  Since $\widehat{G}_{\gamma}(\xi) = \exp(-\gamma \xi^2)$, it follows,
  \begin{align*}
    \|\psi - \psi_0\|_2^2
    &\leq \frac{1}{|\widehat{G}_{\gamma}(u)|^2}
      \int_{\mathbb R^2}| \widehat{W}_{\psi}(z) -
      \widehat{W}_{\psi_0}(z)|^2 |\widehat{G}_{\gamma}(\|z\|)|^2\,dz+ 4 R_n\\
    &\leq 2(C_0 + M_n)e^{2\gamma u^2}
      H^2(P_{\psi}^{\eta},P_{\psi_0}^{\eta}) + 4R_n.
  \end{align*}
  Consequently, when $\|\psi - \psi_0\|_2^2 \geq 8R_n$ we have
  \begin{equation*}
    \|\psi - \psi_0\|_2^2 \leq 4(C_0 +
    M_n)e^{2\gamma u^2}
    H^2(P_{\psi}^{\eta},P_{\psi_0}^{\eta}). \qedhere
  \end{equation*}
\end{proof}

\subsection{Construction of global test functions}
\label{sec:constr-glob-test}

\begin{proof}[Proof of \cref{thm:2}]
  Let $N \equiv N(\sqrt{2}\delta_n^2,\mathcal{F}_n,\|\cdot\|_2)$ denote the
  number of balls of radius $\sqrt{2}\delta_n^2$ and centers in $\mathcal{F}_n$,
  needed to cover $\mathcal{F}_n$. Let $(B_1,\dots,B_N)$ denote the
  corresponding covering with centers $(\psi_1,\dots,\psi_N)$. Now let $J$ be
  the index set of balls $B_j$ with $\|\psi_j - \psi_0\|_2 \geq
  \epsilon_n$. Using \cref{pro:3}, for each of these balls $B_j$ with $j\in J$,
  we can build a test function $\phi_{n,j}$ satisfying
  \begin{equation*}
    P_{\psi_0}^{\eta,n}\phi_{n,j} \leq \exp(-6n \delta_n^2),\qquad
    \sup_{\psi \in B_j}P_{\psi}^{\eta,n}(1 - \phi_{n,j}) \leq \exp(-6n
    \delta_n^2).
  \end{equation*}
  Define the test function $\phi_n := \max_{j\in J}\phi_{n,j}$. Then
  $P_{\psi_0}^{\eta,n}\phi_n \leq \sum_{j\in J} P_{\psi_0}^{\eta,n}\phi_{n,j}
  \leq N\exp(-6n \delta_n^2)$ and
  $P_{\psi}^{\eta,n}(1- \phi_n) \leq \min_{j\in J}\sup_{\psi'\in
    B_j}P_{\psi'}^{\eta,n}(1 - \phi_{n,j}) \leq \exp(-6n \delta_n^2)$ for any
  $\psi \in \mathcal{F}_n$ with
  $\|\psi - \psi_0\|_2 \geq \epsilon_n - \sqrt{2}\delta_n^2$ (recall that
  $\delta_n \ll \epsilon_n$), and hence for any $\psi \in \mathcal{F}_n$ with
  $\|\psi - \psi_0\|_2 \geq \epsilon_n$.
\end{proof}

\section{Proofs for uniform series prior on simplex}
\label{sec:proofs-unif-seri}

\begin{proof}[Proof of \cref{pro:4}]
  From the definition of of $G$ and Hölder's inequality, for $K \geq 0$ integer,
  $Z = KM$ and $(p_{lm})_{(l,m)\in \Lambda_Z}$ in the support of
  $G(\cdot \mid Z)$, we get estimate
  \begin{align*}
    \sum_{(l,m)\in \Lambda_Z}p_{lm}\exp(\beta(l^2+m^2)^{r/2})
    &\leq \sum_{k=1}^K \theta_k \sum_{(l,m)\in \mathcal{I}_k}\eta_{lm}
      \exp(\beta(l^2+m^2)^{r/2})\\
    &\leq \sum_{k=1}^K \theta_k \sqrt{|\mathcal{I}_k|}\exp(\beta k^r M^r),
  \end{align*}
  because $\sum_{k=1}^K\theta_k^2 \geq \theta_1^2 = 1$. The conclusion is direct
  because $\theta_1 = 1$ and $\theta_k \leq \sqrt{2}L \exp(-\beta(k^r-1)M^r)$
  for any $k=2,\dots,K$.
\end{proof}

\begin{proof}[Proof of \cref{pro:9}]
  Let $Z = KM$ for $K > 0$ integer, and
  $(q_{lm})_{(l,m)\in \Lambda_Z} \in \Delta_Z^w(\beta,r,L)$ be arbitrary. For
  any $(l,m) \in \Lambda_Z$, and any sequence
  $(p_{lm})_{(l,m)\in \Lambda_Z} \in \Delta_Z$, let define the unnormalized
  coefficients
  \begin{equation*}
    \widetilde{q}_{lm} := \frac{q_{lm}}{\sqrt{\sum_{(n,p)\in
          \mathcal{I}_1}q_{np}^2}},
    \qquad
    \widetilde{p}_{lm} := \frac{p_{lm}}{
      \sqrt{\sum_{(n,p)\in \mathcal{I}_1}p_{np}^2}},
  \end{equation*}
  Note that
  $\sum_{(l,m)\in \mathcal{I}_1}\widetilde{q}_{lm}^2 = \sum_{(l,m)\in
    \mathcal{I}_1}\widetilde{p}_{lm}^2 = 1$. Moreover, we also have
  \begin{equation*}
    \sum_{(l,m)\in \Lambda_Z}p_{lm}^2 = \sum_{(l,m)\in \Lambda_Z}q_{lm}^2 = 1;
  \end{equation*}
  it turns out that
  \begin{equation*}
    q_{lm} = \frac{\widetilde{q}_{lm}}{
      \sqrt{\sum_{(n,p)\in \Lambda_Z}\widetilde{q}_{np}^2}},\qquad
    p_{lm} = \frac{\widetilde{p}_{lm}}{
      \sqrt{\sum_{(n,p)\in \Lambda_Z}\widetilde{p}_{np}^2}}.
  \end{equation*}
  By the triangle inequality, the two previous expressions of $q_{lm}$, $p_{lm}$
  yield the bound,
  \begin{equation*}
    \sqrt{\sum_{(l,m)\in \Lambda_Z}|q_{lm} - p_{lm}|^2}
    \leq \frac{2 \sqrt{\sum_{(l,m)\in \Lambda_Z}| \widetilde{q}_{lm} -
        \widetilde{p}_{lm}|^2}}{\sqrt{\sum_{(l,m)\in \Lambda_Z}
        \widetilde{q}_{lm}^2}}
    \leq 2 \sqrt{\sum_{(l,m)\in \Lambda_Z}| \widetilde{q}_{lm} -
      \widetilde{p}_{lm}|^2}.
  \end{equation*}
  For any $k=1,\dots,K$, define
  $t_k := \sum_{(l,m)\in \mathcal{I}_k}\widetilde{q}_{lm}^2$ and
  $e_{lm} := \widetilde{q}_{lm} t_k^{-1} \mathbbm 1((l,m)\in
  \mathcal{I}_k)$. Note that by construction we have $t_1 = 1$. With obvious
  definition for $\theta_k$ and $\eta_{lm}$, we have
  \begin{align*}
    \sum_{(l,m)\in \Lambda_Z}|p_{lm} - q_{lm}|^2
    &\leq 2\sum_{k=1}^K\sum_{(l,m)\in \mathcal{I}_k}|\theta_k \eta_{lm} -
      t_ke_{lm}|^2\\
    &\leq 4\sum_{k=1}^Kt_k^2 \sum_{(l,m)\in \mathcal{I}_k}|\eta_{lm} - e_{lm}|^2
      + 4\sum_{k=2}^Z |\theta_k - t_k|^2.
  \end{align*}
  We can choose $M > 0$ large enough to have
  $\sum_{(l,m)\in \mathcal{I}_1}q_{lm}^2 \geq 1/2$; it turns out that
  $\sum_{k=1}^K t_k^2 \leq 2$. Moreover, with $M > 0$ chosen as previously we
  have
  \begin{multline*}
    t_k \exp(\beta k^rM^r) = \sqrt{2}e^{\beta M^r}\sum_{(l,m)\in
      \mathcal{I}_k}q_{lm}
    \exp(\beta (k-1)^rM^r) \\
    \leq \sqrt{2}e^{\beta M^r}\sum_{(l,m)\in \mathcal{I}_k}q_{lm} \exp(\beta
    (l^2+m^2)^{r/2}) \leq \sqrt{2}Le^{\beta M^r},
  \end{multline*}
  thus the coefficients $(t_k)_{k=1}^K$ are in the support of $G(\cdot \mid Z)$.
  By independence structure of the prior, and since
  $\sum_{k=1}^{K}t_k^2 \leq 2$, it suffices to prove that for any $t > 0$,
  \begin{gather}
    \label{eq:15}
    \textstyle
    \prod_{k=1}^KF_k\left( \textstyle \sum_{(l,m)\in \mathcal{I}_k}|\eta_{lm} -
      e_{lm}|^2 \leq t \right) \gtrsim \exp(- c K^{b_1-r} \log t^{-1}),\\
    \label{eq:19}
    \Pr\left(
      \textstyle
      \sum_{k=2}^K|\theta_k - t_k|^2 \leq t \right) \gtrsim \exp(-c'
    K^{b_1 - r}\log t^{-1}),
  \end{gather}
  for some constants $c,c' > 0$. \Cref{eq:15} is automatically satisfied by the
  assumptions on $F_1,F_2,\dots$ in the proposition. \Cref{eq:19} is
  straightforward from the definition of $G(\cdot \mid Z)$.
\end{proof}

\section{Bounding the posterior distribution}
\label{sec:bound-post-distr}

We bound the posterior distribution as follows. Let $\Omega_n$ be the event of
\cref{eq:3}. Then, with the notation $Z_i := (Y_i, \theta_i)$ and
$Z^n=(Z_1,\dots,Z_n)$, for any measurable set $U_n$,
\begin{equation}
  \label{eq:14}
  P_{\psi_0}^{\eta,n}\Pi(U_n\mid Z^n)
  = P_{\psi_0}^{\eta,n}(\Omega_n)\left[I_1^n + I_2^n + I_3^n\right]
  + P_{\psi_0}^{\eta,n}(\Omega_n^c) I_4^n,
\end{equation}
where
\begin{gather*}
  I_1^n := \int_{\Omega_n}\Pi(U_n \cap
  \mathcal{F}_n^c\mid z^n)\,dP_{\psi_0}^{\eta,n}(z^n\mid \Omega_n),\\
  I_2^n := \int_{\Omega_n}\phi_n(z^n)\Pi(U_n \cap
  \mathcal{F}_n\mid z^n)\,dP_{\psi_0}^{\eta,n}(z^n\mid \Omega_n),\\
  I_3^n := \int_{\Omega_n}(1-\phi_n(z^n))\Pi(U_n \cap
  \mathcal{F}_n\mid z^n)\,dP_{\psi_0}^{\eta,n}(z^n\mid \Omega_n),\\
  I_4^n := \int_{\Omega_n^c} \Pi(U_n\mid z^n)\, dP_{\psi_0}^{\eta,n}(z^n\mid
  \Omega_n^c).
\end{gather*}
This decomposition of the expectation for the posterior distribution serves as a
basis for the proof of the next lemma.
\begin{lemma}
  \label{lem:8}
  Let $\delta_n \rightarrow 0$ with $n\delta_n^2 \rightarrow \infty$. Assume
  that there are sets $\mathcal{F}_n \subset \mathcal{F}$ with
  $\Pi(\mathcal{F}_n^c) \leq e^{-6n\delta_n^2}$ and a sequence of test functions
  $(\phi_n)_{n\geq 1}$,
  $\phi_n : (\mathbb R^+ \times [0,2\pi])^n \rightarrow [0,1]$, such that
  $P_{\psi_0}^{\eta,n}\phi_n \rightarrow 0$ and
  $\sup_{\psi\in U_n\cap \mathcal{F}_n} P_{\psi}^{\eta,n}(1 - \phi_n) \leq
  e^{-6n\delta_n^2}$. Also assume that
  $\Pi(B_n(\delta_n)) \gtrsim e^{-n\delta_n^2}$, where $B_n(\delta_n)$ are the
  sets defined in \cref{eq:5}. Then
  $P_{\psi_0}^{\eta,n}\Pi(U_n\mid Z^n) \rightarrow 0$ as $n\rightarrow \infty$.
\end{lemma}
\begin{proof}
  The proof looks like \cite{GhosalGhoshVanDerVaart2000}, with careful
  adaptions. It is obvious that $I_4^n \leq 1$ so that
  $P_{\psi_0}^{\eta,n}(\Omega_n^c)I_4^n \rightarrow 0$ by \cref{thm:1}. With the
  same argument we have that
  $I_{2}^n \leq P_{\psi_0}^{\eta,n}(\Omega_n)^{-1}
  P_{\psi_0}^{\eta,n}\phi_n$. Now we bound $I_3^n$. As usual, recalling that the
  observations are i.i.d we rewrite
  \begin{equation}
    \label{eq:4}
    \Pi(U_n \cap \mathcal{F}_n\mid z^n)
    = \frac{\int_{U_n\cap \mathcal{F}_n}
      \prod_{i=1}^np_{\psi}^{\eta}(y_i,\theta_i) /
      p_{\psi_0}^{\eta}(y_i,\theta_i)\, d\Pi(\psi)}
    {\int \prod_{i=1}^np_{\psi}^{\eta}(y_i,\theta_i) /
      p_{\psi_0}^{\eta}(y_i,\theta_i)\, d\Pi(\psi)}.
  \end{equation}
  We lower bound the integral in the denominator of \cref{eq:4} by integrating
  on the smaller set $B_n$. Consider the events
  \begin{gather*}
    A_n := \Set*{((y_{1},\theta_1),\dots,(y_n,\theta_n)) \given \int_{B_n}
      \prod_{i=1}^n \frac{ p_{\psi}^{\eta}(y_i,\theta_i)}{
        p_{\psi_0}^{\eta}(y_i,\theta_i)}\, \frac{d\Pi(\psi)}{\Pi(B_n)}
      \leq \exp(-4n\delta_n^2) }\\
    C_n := \Set*{((y_{1},\theta_1),\dots,(y_n,\theta_n)) \given \sum_{i=1}^n
      \int_{B_n}\log \frac{p_{\psi_0}^{\eta}(y_i,\theta_i)}{p_{\psi}^{\eta}(y_i,
        \theta_i)} \frac{d\Pi(\psi)}{\Pi(B_n)} \geq 4n\delta_n^2}.
  \end{gather*}
  By Jensen's inequality, we have the inclusion $C_n \subseteq A_n$, thus
  $P_{\psi_0}^{\eta,n}(A_n \mid \Omega_n) \leq P_{\psi_0}^{\eta,n}(C_n \mid
  \Omega_n)$. Moreover, using that the observations are independent, and
  Fubini's theorem, we have
  \begin{multline*}
    P_{\psi_0}^{\eta,n}\left[ \sum_{i=1}^n
      \int_{B_n}\log \frac{p_{\psi_0}^{\eta}(y_i,\theta_i)}{p_{\psi}^{\eta}(y_i,
        \theta_i)} \frac{d\Pi(\psi)}{\Pi(B_n)} \mid \Omega_n\right]\\
    \begin{aligned}
      &= \frac{1}{P_{\psi_0}^{\eta,n}(\Omega_n)} \int_{\Omega_n}\sum_{i=1}^n
      \int_{B_n}\log\frac{p_{\psi_0}^{\eta}(y_i,\theta_i)}{p_{\psi}^{\eta}(y_i,
        \theta_i)} \frac{d\Pi(\psi)}{\Pi(B_n)}\, dP_{\psi_0}^{\eta,n}(
      \textstyle\prod_{j=1}^ndy_jd\theta_j \cap \Omega_n)\\
      &= \frac{n P_{\psi_0}^{\eta}(E_n)^{n-1}}{P_{\psi_0}^{\eta,n}(\Omega_n)}
      \int_{B_n} \left[\int_{E_n} \log
        \frac{p_{\psi_0}(y,\theta)}{p_{\psi}(y,\theta)}\,
        dP_{\psi_0}^{\eta}(dyd\theta)\right]\, \frac{d\Pi(\psi)}{\Pi(B_n)}\\
      &=\frac{n}{P_{\psi_0}^{\eta}(E_n)} \int_{B_n} \left[\int_{E_n} \log
        \frac{p_{\psi_0}(y,\theta)}{p_{\psi}(y,\theta)}\,
        dP_{\psi_0}^{\eta}(dyd\theta)\right]\, \frac{d\Pi(\psi)}{\Pi(B_n)}.
    \end{aligned}
  \end{multline*}
  Likewise, we can bound the variance with respect to
  $P_{\psi_0}^{\eta,n}(\cdot \mid \Omega_n)$, denoted $\mathrm{var}$ for the
  sake of simplicity; with the same arguments as previously,
  \begin{multline*}
    \mathrm{var}\left[ \sum_{i=1}^n \int_{B_n}\log
      \frac{p_{\psi_0}^{\eta}(y_i,\theta_i)}{p_{\psi}^{\eta}(y_i, \theta_i)}
      \frac{d\Pi(\psi)}{\Pi(B_n)}\right]\\
    \begin{aligned}
      &\leq \frac{n}{P_{\psi_0}^{\eta}(E_n)} \int_{E_n} \left( \int_{B_n} \log
        \frac{p_{\psi_0}^{\eta}(y,\theta)}{p_{\psi}^{\eta}(y, \theta)}
        \frac{d\Pi(\psi)}{\Pi(B_n)} \right)^2\,dP_{\psi_0}^{\eta}(y,\theta)\\
      &\leq \frac{n}{P_{\psi_0}^{\eta}(E_n)} \int_{B_n} \left[ \int_{E_n} \left(
          \log \frac{p_{\psi_0}^{\eta}(y,\theta)}{p_{\psi}^{\eta}(y,
            \theta)}\right)^2\,dP_{\psi_0}^{\eta}(y,\theta) \right]
      \frac{d\Pi(\psi)}{\Pi(B_n)},
    \end{aligned}
  \end{multline*}
  From the definition of $B_n$ and because $P_{\psi_0}^{\eta}(E_n) \geq 1/2$ for
  $n$ large enough, we get from Chebychev inequality that for those $n$,
  \begin{equation*}
    P_{\psi_0}^{\eta,n}(A_n \mid \Omega_n)
    \leq P_{\psi_0}^{\eta,n}(C_n \mid \Omega_n)
    \leq \frac{1}{8 n \delta_n^2}.
  \end{equation*}
  Hence,
  \begin{equation*}
    \int_{\Omega_n \cap A_n}
    (1- \phi_n(z^n))\Pi(U_n \cap \mathcal{F}_n \mid z^n)\,
    dP_{\psi_0}^{\eta,n}(z^n\mid \Omega_n) \lesssim
    \frac{P_{\psi_0}^{\eta,n}(A_n)}{P_{\psi_0}^{\eta,n}(\Omega_n)}
    \leq \frac{(n\delta_n^2)^{-1}}{P_{\psi_0}^{\eta,n}(\Omega_n)},
  \end{equation*}
  and,
  \begin{multline*}
    \int_{\Omega_n \cap A_n^c} (1- \phi_n(z^n))\Pi(U_n \cap \mathcal{F}_n\mid
    z^n)\, dP_{\psi_0}^{\eta,n}(z^n\mid \Omega_n)\\
    \begin{aligned}
      &\leq \frac{e^{4n\delta_n^2}}{\Pi(B_n)} \int_{\Omega_n\cap A_n^c}(1 -
      \phi_n(z^n)) \int_{U_n\cap \mathcal{F}_n} \prod_{i=1}^n
      \frac{p_{\psi}^{\eta}(y_i,\theta_i)}{p_{\psi_0}^{\eta}(y_i,\theta_i)}\,
      d\Pi(\psi)dP_{\psi_0}^{\eta,n}(z^n\mid \Omega_n)\\
      &= \frac{e^{4n\delta_n^2}}{\Pi(B_n)} \int_{U_n\cap \mathcal{F}_n}
      \int_{\Omega_n\cap A_n^c}(1 - \phi_n(z^n))\prod_{i=1}^n
      \frac{p_{\psi}^{\eta}(y_i,\theta_i)}{p_{\psi_0}^{\eta}(y_i,\theta_i)}\,
      dP_{\psi_0}^{\eta,n}(z^n\mid \Omega_n) d\Pi(\psi)\\
      &\leq \frac{e^{4n\delta_n^2}\Pi(U_n\cap \mathcal{F}_n) }{\Pi(B_n)}
      \frac{\sup_{\psi \in U_n \cap \mathcal{F}_n}
        P_{\psi}^{\eta,n}(1-\phi_n)}{P_{\psi_0}^{\eta,n}(\Omega_n)}.
    \end{aligned}
  \end{multline*}
  where the third line follows from Fubini's theorem. Combining the last two
  results yields $P_{\psi_0}^{\eta,n}(\Omega_n)I_3^n \rightarrow 0$. The bound
  on $I_1^n$ follows exactly the same lines as the bound on $I_3^n$ \citep[see
  also][]{GhosalGhoshVanDerVaart2000}.
\end{proof}

\section{Remaining proofs and auxiliary results}
\label{sec:auxiliary-results}

\begin{lemma}
  \label{lem:1}
  Let $\psi,\psi_0\in \mathbb S^2(\mathbb R)$. Then,
  $H^2(P_{\psi}^{\eta},P_{\psi_0}^{\eta}) \leq \sqrt{2}H(P_{\psi},P_{\psi_0})
  \leq \sqrt{2}\|\psi - \psi_0\|_2$. Moreover, we also have that
  $H(P_{\psi}(\cdot \mid \theta), P_{\psi_0}(\cdot \mid \theta)) \leq \|\psi -
  \psi_0\|_2$ for all $\theta \in [0,\pi]$.
\end{lemma}
\begin{proof}
  First, we recall that
  $p_{\psi}^{\eta}(y,\theta) = [p_{\psi}(\cdot,\theta)*G_{\gamma}](y)$. The same
  holds for $p_{\psi_0}^{\eta}$. Then using that the square Hellinger distance
  is bounded by the total variation distance, which is in turn bounded by the
  Hellinger distance,
  \begin{align*}
    H^2(P_{\psi}^{\eta},P_{\psi_0}^{\eta})
    &\leq \iint
      |[p_{\psi}(\cdot,\theta)*G_{\gamma}](y)
      - [p_{\psi_0}(\cdot,\theta)*G_{\gamma}](y)|\,dyd\theta\\
    &\leq \|G_{\gamma}\|_1\iint |p_{\psi}(x,\theta) - p_{\psi_0}(x,\theta)|\,
    dxd\theta \leq \sqrt{2} H(P_{\psi},P_{\psi_0}),
  \end{align*}
  where the second line follows from Young's inequality.  Now let $\theta \ne 0$
  and $\theta \ne \pi/2$. Using that $|x|-|y| = |x-y+y|-|y| \leq |x-y|$ for all
  $x,y\in \mathbb C$, it holds from \cref{eq:1} that,
  \begin{multline*}
    \sqrt{p_{\psi}(x,\theta)} - \sqrt{p_{\psi_0}(x,\theta)}\\
    \leq \frac{1}{2\pi\sqrt{|\sin\theta|}}
      \left|
      \int_{-\infty}^{+\infty} \left(\psi(z) - \psi_0(z)\right)
      \,\exp\left( i \frac{\cot\theta}{2} z^{2} - i
      \frac{x}{\sin\theta} z \right)\, dz
      \right|.
  \end{multline*}
  On almost recognize the expression of the square-root of a density in the rhs
  of the last equation. Indeed, it is not because $\psi - \psi_0$ is not
  normalized in $L^2$. But, letting
  $\psi_v := (\psi - \psi_0)/\|\psi - \psi_0\|_{2}$,
  \begin{equation}
    \label{eq:24}
    \left(\sqrt{p_{\psi}(x,\theta)} - \sqrt{p_{\psi_0}(x,\theta)}\right)^2
    \leq p_v(x,\theta)\|\psi - \psi_0\|_2^2.
  \end{equation}
  One can show easily that the same bound holds when $\theta = 0$ or
  $\theta = \pi/2$ (although it is even not necessary). The conclusion of the
  lemma then follows from the definition of the Hellinger distance and the fact
  that $p_{v}$ is a probability density. The results for conditional densities
  is immediate from \cref{eq:24} since
  $p_{\psi}(x\mid \theta) = \pi p_{\psi}(x,\theta)$ for any
  $\psi \in \mathbb S^2(\mathbb R)$.
\end{proof}

\begin{proposition}
  \label{pro:2}
  There exists $n_0$ such that for all $n \geq n_0$ and all
  $\psi \in \mathcal{M}_n(Z,U)$ it holds
  $P_{\psi}(|X| \leq D_n^{\beta,r} \mid \theta) \geq 1/2$ for all
  $\theta \in [0,\pi]$.
\end{proposition}
\begin{proof}
  It suffices to write that,
  \begin{align*}
    P_{\psi_0}(|X| \leq D_n^{\beta,r} \mid \theta)
    &\leq \int_{[-D_n^{\beta,r},+D_n^{\beta,r}]}p_{\psi}(x\mid \theta)\,dx +
      \int_{\mathbb R}|p_{\psi}(x\mid \theta) - p_{\psi_0}(x\mid
      \theta)|\,dx\\
    &\leq P_{\psi}(|X| \leq D_n^{\beta,r}\mid \theta)
      + \sqrt{2}H(P_{\psi}(\cdot\mid \theta), P_{\psi_0}(\cdot\mid \theta)).
  \end{align*}
  By \cref{lem:1},
  $\sqrt{2} H(P_{\psi}(\cdot \mid \theta), P_{\psi_0}(\cdot \mid \theta)) \leq
  1/4$ for all $\psi \in \mathcal{M}_n$ if $n$ is large enough. Moreover, is $n$
  is sufficiently large, we also have
  $P_{\psi_0}(|X| \leq D_n^{\beta,r} \mid \theta) \geq 3/4$, concluding the
  proof.
\end{proof}

% \section*{Acknowledgments}

% \input{6-acknowledgements}

% The bibliography
% ------------------------------------------------------------------------------
\bibliographystyle{abbrvnat}
\bibliography{bib/biblio-paper}

\end{document}